\documentclass[11pt]{article} 

\usepackage{amsmath}
\usepackage{amssymb}
\usepackage{color}
\usepackage[T1]{fontenc}
\usepackage[latin1]{inputenc}
\def\r{\rightarrow}

\usepackage{enumerate}

\newcommand{\fdem}{\hspace*{\fill}~$\Box$\par\endtrivlist\unskip}

\newcommand{\E}{\mathbb{E}}     
\renewcommand{\P}{\mathbb{P}}     
\renewcommand{\L}{\mathbb{L}}

\newcommand{\N}{\mathbb{N}}     
\newcommand{\Z}{\mathbb{Z}}
\newcommand{\R}{\mathbb{R}}     
     
\newcommand{\C}{\mathbb{C}} 
 
\newcommand{\X}{\mathbb{X}}
\newcommand{\V}{\mathbb{V}}

\renewcommand{\dim}{\mathop{\rm dim}}
\renewcommand{\ker}{\mathop{\rm Ker}}

\renewcommand{\r}{\mathop{\rightarrow}}

\newcommand{\cB}{\mbox{$\cal B$}}
\newcommand{\cC}{\mbox{$\cal C$}}

\newcommand{\cL}{\mbox{$\cal L$}}
\newcommand{\cM}{\mbox{$\cal M$}}
\newcommand{\cN}{\mbox{$\cal N$}}

\newcommand{\cR}{\mbox{$\cal R$}}

\newcommand{\cX}{\mbox{$\cal X$}}
\newcommand{\cU}{\mbox{$\cal U$}}
\newcommand{\cV}{\mbox{$\cal V$}}

\newtheorem{theo}{Theorem}
\newtheorem{pro}{Proposition}
\newenvironment{proof}[1]{\textit{Proof#1.\,}}{\fdem}
\newtheorem{lem}{Lemma}
\newtheorem{rem}{Remark}
\newtheorem{ex}{Example}
\newtheorem{cor}{Corollary}
\newtheorem{defi}{Definition}
\newtheorem{alem}{Lemma}[section]
\newtheorem{apro}[alem]{Proposition}

\title{Quasi-compactness of Markov kernels on weighted-supremum spaces and geometrical ergodicity}

\author{Denis GUIBOURG, Loïc HERV\'E, and James LEDOUX \footnote{INSA de Rennes, IRMAR, F-35000, France; CNRS, UMR 6625, Rennes, F-35000, France; Université Européenne de Bretagne, France.
 Denis.Guibourg@wanadoo.fr, \{Loic.Herve,James.Ledoux\}@insa-rennes.fr}
}

\begin{document}

\maketitle

\begin{abstract}
Let $P$ be a Markov kernel on a measurable space $\X$ and let $V:\X\r[1,+\infty)$. We provide various assumptions, based on drift conditions, under which $P$ is quasi-compact on the weighted-supremum Banach space $(\cB_V,\|\cdot\|_V)$ of all  the measurable functions $f : \X\r\C$ such that $\|f\|_V  := \sup_{x\in \X} |f(x)|/V(x) < \infty$. Furthermore we give bounds  for the essential spectral radius of $P$. Under additional assumptions, these results allow us to derive the convergence rate of $P$ on $\cB_V$, that is the geometric rate of convergence of the iterates $P^n$ to the stationary distribution in operator norm. Applications to discrete Markov kernels and to iterated function systems are presented. 
\end{abstract}

\begin{center}

AMS subject classification : 60J10; 47B07

Keywords : Markov chain, drift condition, essential spectral radius, convergence rate, birth-death Markov chains. 
\end{center}
\newpage
{\scriptsize
\tableofcontents
}

\newpage

\section{Introduction} \label{intro}

Let $P$ be a Markov kernel on a measurable space $(\X,\cX)$. Let us introduce the weighted-supremum Banach space $(\cB_V,\|\cdot\|_V)$ composed of  measurable functions $f : \X\r\C$ such that 
$$\|f\|_V  := \sup_{x\in \X} \frac{|f(x)|}{V(x)} < \infty$$
where $V:\X\r[1,+\infty)$. Let $(\cB_0,\|\cdot\|_0)$ be the usual Banach space composed of all the bounded measurable functions $f : \X\r\C$ equipped with the supremum norm $\|f\|_0:=\sup_{x\in\X}|f(x)|$. 

The first purpose of the paper is to study the quasi-compactness of $P$ on $\cB_V$ with a control of its essential spectral radius $r_{ess}(P)$. Recall that $r_{ess}(P)$ is the infimum bound of the positive real numbers $r_0$ for which the following property holds:  the spectral values of $P$ of modulus greater than $r_0$ are finitely many eigenvalues having a finite-dimensional characteristic space. $P$ is said to be quasi-compact on $\cB_V$ if $r_{ess}(P)$ is strictly less than the spectral radius of $P$ (see Section~\ref{sec-mino} for details).  
The second purpose of the paper is to specify the link between quasi-compactness and the so-called $V$-geometric ergodicity \cite{MeyTwe93}, namely with the convergence of $P^n$ to $\pi$ in operator norm on $\cB_V$, where $\pi$ denotes the $P$-invariant probability measure. In this case, we are interested in finding upper bounds for the convergence rate $\rho_V(P)$ defined by 
\begin{equation} \label{Def_rhoV}
	\rho_V(P) := \inf\big\{ \rho \in (0,1),\sup_{\|f\|_V\leq1}\|P^nf-\pi(f)\|_V = O(\rho^n)\big\}.
\end{equation}
Finally the third purpose of the paper is to derive the $V$-geometric ergodicity of $P$, with a control of $\rho_V(P)$, from the strong ergodicity property with respect to some Lipschitz-weighted spaces.

 Note that this paper is not directly concerned with the essential spectral radius or the convergence rate of Markov chains either with respect to  the Lebesgue space $\L^2(\pi)$ as studied for instance in \cite[Sec.~5]{Wu04} for general Markov kernels and in \cite{AtcPer07} for Hastings and Metropolis algorithms using operator methods (see \cite[Section~2]{FerHerLed10} for an overview), or with respect to $\cB_V$ with a bounded function $V$, that is for uniformly ergodic Markov chains as investigated for instance in \cite[Th.~3.10]{Wu04} and \cite[Cor.~IV.1]{Hen07}. Mention that the paper \cite{Wu04} also deals with the essential spectral radius and the convergence rate of iterates  of $P$ acting on $\cB_V$. Actually Wu's article is the closest work to ours. Precise comparisons between our results and those of \cite{Wu04} are included throughout the paper but the core is in Subsection~\ref{Sec_Wu}.

Let us give an account of the main results of the paper in regards to our objectives. Under irreducibility and aperiodicity assumptions, it is well-known that the $V$-geometric ergodicity holds under the following drift condition:  
\begin{equation} \label{inequality-drift}
 \exists \varrho\in(0,1),\ \exists M\in(0,+\infty),\quad PV \leq \varrho \, V + M\, 1_S, \tag{\textbf{D}}
\end{equation}
where $S\subset\X$ satisfies the minorization condition 
\begin{equation} \label{small} 
 \forall x\in \X,\ \forall A\in\cX,\ \ \ P(x,A) \geq \nu(A)\, 1_S(x), \tag{\textbf{S}}
\end{equation}
for some positive measure $\nu$ on $(\X,\cX)$ (see \cite{MeyTwe93}). In Theorem~\ref{main} (Subsection~\ref{subsec-qc-drift}), without assuming any irreducibility or aperiodicity conditions, the quasi-compactness of $P$ on $\cB_V$ is proved under Conditions~(\ref{inequality-drift})-(\ref{small}). This is an expected result, already obtained in \cite{Hen06,Hen07}. We provide a simple and short proof of Theorem~\ref{main} which enables to well understand why the drift condition implies good spectral properties of $P$ on $\cB_V$. Furthermore we obtain the following upper bound for $r_{ess}(P)$ which is more explicit than in \cite{Hen06}:  
\begin{equation} \label{intro-main-r-ess}
r_{ess}(P) \leq \frac{\varrho\, \nu(1_\X) + \tau}{\nu(1_\X)+\tau} \quad \text{with } \tau:=\max(0,M-\nu(V)).
\end{equation} 
\indent In Theorem~\ref{pro-qc-bis} (Subsection~\ref{sub-suf-K}), assuming that,  for some $\ell\geq 1$, $P^\ell$ is a compact operator from $\cB_0$ to $\cB_V$, $P$ is shown to be power-bounded and quasi-compact on $\cB_V$ under the following weak drift condition 
\begin{equation} \label{cond-D}
\exists N\in\N^*,\ \exists d\in(0,+\infty),\ \exists \delta\in(0,1),\quad P^NV \leq \delta^N\, V + d\, 1_{\X}. \tag{\textbf{WD}}
\end{equation} 
Such a condition with $N=1$ has been introduced in \cite[Lem.~15.2.8]{MeyTwe93} as an alternative to the drift condition \cite[(V4)]{MeyTwe93} under suitable assumption on $V$. Under Condition~(\ref{cond-D}), let us define the real number $\delta_V(P)\in(0,1)$ as the infimum of the real numbers $\delta\in[0,1)$ such that we have (\ref{cond-D}):
\begin{equation} \label{def-hat-rho}
\delta_V(P) := \inf\big\{\delta\in[0,1) : \exists N\in\N^*,\, \exists d\in(0,+\infty),\ P^NV \leq \delta^N\, V + d\, 1_{\X}\big\}.
\end{equation}
Then the upper bound obtained in Theorem~\ref{pro-qc-bis} for $r_{ess}(P)$ is more precise than (\ref{intro-main-r-ess}), that is: 
 $$r_{ess}(P) \leq \delta_V(P).$$  
The key idea to prove Theorem~\ref{pro-qc-bis} is that Condition~(\ref{cond-D}) yields a Doeblin-Fortet inequality on the dual of $\cB_V$. This fact has been already used in \cite{FerHerLed11} to study regular perturbations of $V$-geometrically ergodic Markov chains. Under assumptions based on sophisticated parameters $\beta_w(\cdot)$ and $\beta_\tau(\cdot)$ for measure of non-compactness of $P$, Wu presented in \cite[Th.~3.11]{Wu04} a formula for $r_{ess}(P)$ involving equivalent functions to $V$. The assumptions, the conclusion and the proof of Wu's result are different from ours, as explained in Subsection~\ref{Sec_Wu}. The question to know if the equality $r_{ess}(P) = \delta_V(P)$ holds true under the hypotheses of Theorem~\ref{pro-qc-bis} is open.  However, by combining our Theorem~\ref{pro-qc-bis} and Wu's result, we prove in Subsection~\ref{Sec_Wu} that the answer to the previous question is positive in many situations. In particular we have $r_{ess}(P) = \delta_V(P)$ in all the examples of our paper. 

When the Markov kernel $P$ has an invariant probability distribution, the connection between the $V$-geometric ergodicity and the quasi-compactness of $P$ is recalled in Theorem~\ref{CNS-qc-Vgeo} 
(Subsection~\ref{sec-qc-V-geo}). Namely, $P$ is $V$-geometrically ergodic if and only if $P$ is a power-bounded quasi-compact operator on $\cB_V$ for which $\lambda=1$ is a simple eigenvalue and the unique eigenvalue of modulus one. In this case, if $\cV$ denotes the set of all the eigenvalues $\lambda$ of $P$ such that $r_{ess}(P)<|\lambda|<1$, then the convergence rate $\rho_V(P)$ is given by: $$\rho_V(P)=r_{ess}(P)\ \text{ if }\ \cV=\emptyset \quad \text{ and } \quad \rho_V(P)=\max\{|\lambda|,\, \lambda\in\cV\}\ \text{ if }\ \cV\neq\emptyset.$$
This result is valid for any quasi-compact operator, however we have not found such an explicit result in the literature on $V$-geometric ergodicity. 

Theorem~\ref{pro-tail-fct-propre} proved in Subsection~\ref{subsec-taille-fct-propre} is of great interest to investigate the eigenvalues of modulus one and the above set $\cV$ in order to obtain the $V$-geometric ergodicity of $P$ and,  more importantly, an upper bound for $\rho_V(P)$ from Theorem~\ref{CNS-qc-Vgeo}. Namely, under Condition (\ref{cond-D}), for any $\lambda\in\C$ such that $\delta \leq |\lambda| \leq 1$ where $\delta$ is given in (\ref{cond-D}), and for any $p\in\N^*$, we obtain with $\beta(\lambda) := \ln|\lambda|/\ln\delta$: 

\begin{equation} \label{intro-siez-fc}
f\in\cB_V\cap\ker(P-\lambda I)^p\ \Rightarrow\ \exists c\in(0,+\infty),\ |f| \leq c\, (\ln V)^{p(p-1)/2}\, V^{\beta(\lambda)}. 
\end{equation}  
In particular, if $\lambda$ is an eigenvalue such that $|\lambda|=1$, then any associated eigen-function $f$ is bounded on $\X$. By contrast, if $|\lambda|$ is close to $\delta_V(P)$, then $|f| \leq c\, V^{\beta(\lambda)}$ with $\beta(\lambda)$ close to~1. 

In Section~\ref{sub-sec-countable}, applications of Theorems~\ref{pro-qc-bis}-\ref{pro-tail-fct-propre} to discrete Markov chains are presented. When  $X:=\N$ and $\lim_nV(n)=+\infty$, any Markov kernel $P$ is compact from $\cB_0$ to $\cB_V$, and Theorem~\ref{pro-qc-bis}-Theorem~\ref{CNS-qc-Vgeo} are then specially relevant: if $P$ satisfies Condition~(\ref{cond-D}), then $P$ is  power-bounded and quasi-compact on $\cB_V$; if in addition $P$ is irreducible and aperiodic, then $P$ is $V$-geometrically ergodic.  In Subsection~\ref{sub-basic-ex-revis}, Property~(\ref{intro-siez-fc}) is used to compute the convergence rate $\rho_V(P)$ for birth-and-death Markov chains. 

 Section~\ref{sect-strong-ergo} is devoted to $V$-geometrical ergodicity of iterated function systems (IFS). The ideas developed in this section are based on Lipschitz contractive properties of $P$ as in \cite[Sect.~7.2]{Wu04}. More precisely, in \cite[Sect.~7.2]{Wu04} the contractive properties are expressed in terms of Wassertein distance. Ours are expressed in terms of moment/contraction conditions, called $(\cC_{a})$ (for some ${a}\in[1,+\infty)$), which are classical for IFS, see \cite{Duf97,Ben98,DiaFre99}. Under Conditions~$(\cC_{a})$ and our compactness assumption on  $P^\ell:\cB_0\r\cB_V$ for some $\ell\geq1$, the same precise bounds on $r_{ess}(P)$ and $\rho_V(P)$ as in \cite[Sect.~7.2]{Wu04} are obtained for IFSs in Corollary~\ref{pro-D-K1}. The others statements of Section~\ref{sect-strong-ergo} show that, in certain cases, Condition~$(\cC_{a})$ can be used directly to obtain further interesting rates of convergence of IFSs with explicit constants.

To the best of our knowledge, Theorems~\ref{main}, \ref{pro-qc-bis} and \ref{pro-tail-fct-propre} of Section~\ref{sec-mino} are new. Moreover Theorem~\ref{pro-qc-bis} and Wu's result \cite[Th.~3.11]{Wu04} are complementary since their combination provides the expected formula $r_{ess}(P) = \delta_V(P)$ under general assumptions. As in \cite{Wu04}, the bounds on $\rho_V(P)$ are derived from those on $r_{ess}(P)$ by using Theorem~\ref{CNS-qc-Vgeo}, but here we take advantage of Theorem~\ref{pro-tail-fct-propre} to study the eigenvalues $\lambda$ of $P$ such that $\delta_V(P)<|\lambda|\leq1$. This approach is original and often provides the exact value of $\rho_V(P)$. Most of bounds on $\rho_V(P)$ obtained in Section~\ref{sub-sec-countable} are new. Of course this method can only be used for Markov kernels $P$ such that $P^\ell$ is compact from $\cB_0$ to $\cB_V$ for some $\ell\geq1$. Classical instances of $V$-geometrically ergodic Markov kernels concern the discrete state-space, the  autoregressive models on $\X=\R^q$ with absolutely continuous noise with respect to the Lebesgue measure, and finally the MCMC algorithms. Our compactness assumption is fulfilled for the two first instances, see Section~\ref{sub-sec-countable} and Subsection~\ref{ex1-auto}. Unfortunately it does not hold in general for non-discrete Markov kernels arising from Hastings and Metropolis algorithms. Concerning the last issue, we refer to the works \cite{Bax05,Lun97,LunTwe96,MenTwe96,MeyTwe94,RobTwe99} where the convergence rate $\rho_V(P)$ is investigated by probabilistic methods. The best rates are obtained in \cite{LunTwe96} under the stochastic monotonicity assumption for $P$ which cannot be compared with our compactness assumption. 

Throughout the paper, when the function $V\equiv V_c$ depends on some parameter $c$, we use the notation $\cB_c \equiv \cB_{V_c}$ so that $\cB_c$ may stand for different sets from section to section. 

\section{Quasi-compactness on $\cB_V$ and $V$-geometric ergodicity} \label{sec-mino}
 Let $(\cB,\|\cdot\|)$ be a complex Banach space, and let $L$ be a bounded linear operator on $\cB$ with positive spectral radius $r(L):=\lim_n\|L^n\|^{1/n}$, where $\|\cdot\|$ also stands for the operator norm on $\cB$. For the sake of simplicity, we assume that $r(L):=1$ (if not, replace $L$ with $r(L)^{-1}L$). The restriction of $L$ to a $L$-invariant subspace $H$ is denoted by $L_{|H}$, and $I$ stands for the identity operator on $\cB$. 
 
 The simplest definition of quasi-compactness is the following (compare the definition below with the reduction of matrices or compact operators).
\begin{defi} \label{def-q-c} 
$L$ is quasi-compact on $\cB$ if there exist $r_0\in(0,1)$ and $m\in\N^*$, $\lambda_i\in\C$, $p_i\in\N^*$ ($\, i=1,\ldots,m$) such that: 
\begin{subequations}
\begin{equation} \label{noyit}
\cB = \overset{m}{\underset{i=1}{\oplus}} \ker(L - \lambda_i I)^{p_i}\, \oplus H, 
\end{equation}
where the $\lambda_i$'s are such that 
\begin{equation} \label{noyit-lambda}
|\lambda_i| \geq r_0\quad \text{ and } \quad 1\le \dim\ker(L-\lambda_i I)^{p_i} < \infty,
\end{equation}
and $H$ is a closed $L$-invariant subspace such that 
\begin{equation} \label{noyit-H}
\sup_{h\in H,\, \|h\|\leq1}\|L^nh\| = O({r_0}^n).
\end{equation}
\end{subequations}
\end{defi}

Concerning the essential spectral radius of $L$, denoted by $r_{ess}(L)$, here it is enough to have in mind that, if $L$ is quasi-compact on $\cB$, then we have (see for instance \cite{Hen93})
$$r_{ess}(L) = \inf\big\{r_0\in(0,1) \text{ s.t.~we have (\ref{noyit}) (\ref{noyit-lambda}) (\ref{noyit-H})} \big\}.$$
It is also well-known (e.g.~see \cite{Nev64,Kre85}) that $r_{ess}(L)$ is defined by 
\begin{equation} \label{ray-ess}
r_{ess}(L) := \lim_n\big(\inf \|L^n-K\|\big)^{\frac{1}{n}} 
\end{equation}
where the infimum is taken over the ideal of compact operators $K$ on  $\cB$. Consequently $L$ is quasi-compact if and only if there exist some $n_0\in\N^*$ and some compact operator $K_0$ on $\cB$ such that 
$r(L^{n_0}-K_0)<1$. Under the previous condition we have 
\begin{equation} \label{ray-ess-bis}
r_{ess}(L) \leq (r(L^{n_0}-K_0))^{1/n_0}.
\end{equation}
Indeed, for all $k\geq1$  we have $\|(L^{n_0}-K_0)^k\|^{1/(n_0k)} = \|L^{n_0k} - K_k\|^{1/(n_0k)}$ with some compact operator $K_k$ on $\cB$. Then (\ref{ray-ess}) gives:  
$r_{ess}(L) \leq \lim_k \|(L^{n_0}-K_0)^k\|^{1/(n_0k)} = (r(L^{n_0}-K_0))^{1/n_0}$. 
Finally, for any $\ell\geq1$, since 
$\lim_n(\inf \|L^n-K\|)^{1/n} = \lim_k(\inf \|L^{\ell k} - K\|)^{1/(\ell k)}$, we obtain  
\begin{equation} \label{ray-ess-puissance}
r_{ess}(L) = (r_{ess}(L^\ell))^{1/\ell}.
\end{equation}

Throughout the paper, we consider a function $V : \X\r[1,+\infty)$ and a Markov kernel $P$ on $(\X,\cX)$ such that $PV/V$ is bounded on $\X$ (i.e. $\|PV\|_{V}<\infty$). So $P$ continuously acts on~$\cB_V$. 

\subsection{Quasi-compactness on $\cB_V$ under the drift condition} \label{subsec-qc-drift}

\begin{theo} \label{main}
Let us assume that the Conditions~\emph{(\ref{inequality-drift})-(\ref{small})} in Introduction hold true. Then $P$ is a power-bounded quasi-compact operator on $\cB_V$ with 
\begin{equation} \label{first-r-ess}
 r_{ess}(P) \leq \frac{\varrho\, \nu(1_\X) + \tau}{\nu(1_\X)+\tau} \quad \text{with } \tau:=\max(0,M-\nu(V)).
\end{equation}
\end{theo}

The proof of Theorem~\ref{main} is based on the next lemma. 
\begin{lem} \label{Q-eta}Let $Q(x,dy)$ be a Markov kernel on $(\X,\cX)$ having a continuous action on $\cB_V$ (i.e.~$\|QV\|_{V}<\infty$) such that $Q = A + B$ for some nonnegative bounded linear operators $A$ and $B$ on $\cB_V$. Let $r(B)$ denote the spectral radius of $B$ which is assumed to be positive. \newline
Then, there exists a nontrivial nonnegative continuous linear form $\eta$ on $\cB_V$ such that $\eta \circ B= r(B)\, \eta $ and $\eta(A1_\X) = (1-r(B))\eta(1_\X)$.  
\end{lem}
\begin{proof}{}
Since $B\geq0$ and $r:=r(B)>0$, we know from \cite[App., Cor.2.6]{Sch71} that there exists a nontrivial nonnegative  
continuous linear form $\eta$ on $\cB_V$ such that $\eta \circ B = r\, \eta$ 
(see also Remark~\ref{eta-gene} and Appendix~\ref{B}). From $Q = A + B$, we have $\eta\circ Q = \eta\circ A + r\, \eta$, thus $\eta(Q1_\X) = \eta(1_\X) = \eta(A1_\X) + r\, \eta(1_\X)$. Hence $\eta(A1_\X) = (1-r)\eta(1_\X)$.  
\end{proof}
\begin{proof}{ of Theorem~\ref {main}}
Condition~(\ref{inequality-drift}) implies that $PV \leq \varrho\, V + M\,1_\X$. Iterating this inequality easily ensures that $\sup_k\|P^{k}V\|_V<\infty$, that is $P$ is power-bounded. Then, from $P1_\X=1_\X$ and $1_\X\in\cB_V$, we have $r(P)=1$. Moreover, since $\|PV\|_{V}<\infty$, we deduce from (\ref{small}) that $\nu(V)<\infty$. Thus we can define the following rank-one operator on $\cB_V$: $Tf := \nu(f)\, 1_S$. Let $R := P-T$. From $T\geq0$ and from (\ref{small}), it follows that $0\leq R \leq P$, so $r(R) \leq 1$. Let us set $r:=r(R)$. If $r=0$, then $P$ is quasi-compact with $r_{ess}(P) =0$ from (\ref{ray-ess-bis}). Now assume that $r\in(0,1]$. Then, from Lemma~\ref{Q-eta}, there exists a nontrivial nonnegative continuous linear form $\eta$ on $\cB_V$ such that $\eta \circ R = r\, \eta$ and $\eta(T1_\X) = (1-r)\eta(1_\X)$, from which we deduce that $$\eta(1_S) = \frac{(1-r)\eta(1_\X)}{\nu(1_\X)} \leq \frac{(1-r)\eta(V)}{\nu(1_\X)}.$$
Next, we have $RV = PV - TV = PV - \nu(V) 1_S \leq \varrho V + M1_S - \nu(V) 1_S = \varrho\, V + (M-\nu(V))\, 1_S$. Hence, setting $\tau := \max(0,M-\nu(V))\ge 0$,
\begin{equation} \label{11}
r\, \eta(V) = \eta(RV) \leq \varrho\, \eta(V) + \tau\, \eta(1_S) \leq \varrho\, \eta(V) + \tau\frac{(1-r)\eta(V)}{\nu(1_\X)}.
\end{equation} 
Since $\eta\neq 0$, we have $\eta(V) > 0$, and since $\varrho\in(0,1)$, we cannot have $r=1$. Thus $r\in(0,1)$, and $P$ is quasi-compact from (\ref{ray-ess-bis}) with $r_{ess}(P) \leq r_{ess}(R) = r$. Then Inequality~(\ref{first-r-ess}) is deduced from (\ref{11}). 
\end{proof}
\begin{rem} \label{rk-set-itere}
If Conditions~\emph{(\ref{inequality-drift})-(\ref{small})} are fulfilled for some iterate $P^N$ in place of $P$ (with parameters $\varrho_N<1$, $M_N>0$ and positive measure $\nu_N(\cdot)$), then the conclusions of Theorem~\ref {main} hold true with (\ref{first-r-ess}) replaced by 
$$r_{ess}(P) = r_{ess}(P^N)^{1/N} \leq \left(\frac{\varrho_N \nu_N(1_\X) + \tau_N}{\nu_N(1_\X)+\tau_N}\right)^{\frac{1}{N}}
\quad \text{where } \tau_N:=\max(0,M_N-\nu_N(V)).$$ 
\end{rem}
\begin{rem} \label{eta-gene}
The proof of Lemma~\ref{Q-eta} is based on the  following result \cite[App., Cor.2.6]{Sch71}: if $L$ is a positive operator on a Banach lattice $\cB$ whose positive cone is normal and has interior points, then there exists a nontrivial nonnegative continuous linear form $e'$ on $\cB$ such that $e'\circ L = r(L)\, e'$. In fact $\cB_V$ is the simplest (and generic) example of Banach lattices satisfying the last conditions, and we give in Appendix~\ref{B} a proof of the previous statement in this special case. Mention that this result also provides that the quasi-compactness of $P$ on $\cB_V$ is equivalent to the mean ergodicity with finite rank limit projection (see \cite{Her08}, see also \cite{Lin75,Lin78}). 
\end{rem}

\subsection{Quasi-compactness on $\cB_V$ under the weak drift condition~(\ref{cond-D})} \label{sub-suf-K}
Recall that $(\cB_0,\|\cdot\|_0)$ denotes the Banach space of all the bounded measurable functions $f : \X\r\C$, equipped with the supremum norm $\|f\|_0:=\sup_{x\in\X}|f(x)|$, and that $\delta_V(P)$ is the infimum  of the real numbers $\delta\in[0,1)$ such that we have (\ref{cond-D}) (see (\ref{def-hat-rho})). 
\begin{theo} \label{pro-qc-bis} 
If Condition~\emph{(\ref{cond-D})} holds true and if $P^\ell : \cB_0\r\cB_V$ is compact for some $\ell\geq1$, then $P$ is a power-bounded quasi-compact operator on $\cB_V$, and we have $$r_{ess}(P) \leq \delta_V(P).$$ 
\end{theo}
\begin{proof}{} 
Iterating (\ref{cond-D}) shows that $P$ is power-bounded on $\cB_V$ (proceed as in the beginning of the proof of Theorem~\ref {main}). Since $\delta_V(P) = (\delta_V(P^\ell))^{1/\ell}$ and $r_{ess}(P) = (r_{ess}(P^\ell))^{1/\ell}$ (see (\ref{ray-ess-puissance})), we only consider the case $\ell:=1$, that is $P : \cB_0\r\cB_V$ is compact.

Now let $(\cB_V',\|\cdot\|_V)$ (resp.~$(\cB_0',\|\cdot\|_0)$) denote the dual space of $\cB_V$ (resp.~of $\cB_0$). Note that we make a slight abuse of notation in writing again $\|\cdot\|_V$ and $\|\cdot\|_0$ for the dual norms. Let $P^*$ denote the adjoint operator of $P$ on $\cB_V'$. In fact, we prove that $P^*$ is a quasi-compact operator on $\cB_V'$ with $r_{ess}(P^*) \leq \delta_V(P)$, so that $P$ satisfies the same properties on $\cB_V$. Since $P : \cB_0\r\cB_V$ is assumed to be compact, then so is $P^* : \cB_V'\r\cB_0'$. Moreover $P^*$ satisfies a Doeblin-Fortet inequality from Lemma~\ref{lem-D-F} below. Then we deduce from Lemma~\ref{lem-D-F} and \cite{Hen93} that $P^*$ is a quasi-compact operator on $\cB_V'$, with $r_{ess}(P^*) \leq \delta$ for any $\delta\in(\delta_V(P),1)$, so that $r_{ess}(P^*) \leq \delta_V(P)$. 
\end{proof}

For the sake of simplicity we consider the same usual bracket notation $\langle\cdot,\cdot\rangle$ in both  $\cB_V'\times\cB_V$ and $\cB_0'\times\cB_0$.  Recall that $\cB_V, \cB_0$ are Banach lattices, so are $\cB_V'$, $\cB_0'$. For each $g'\in\cB_V'$ (resp.~$g'\in\cB_0'$), one can define the modulus $|g'|$ of $g'$ in $\cB_V'$ (resp.~in $\cB_0'$), see \cite{Sch71}. For the next arguments, it is enough to have in mind 
that $g'$ and $|g'|$ have the same norm in $\cB_V'$ (resp.~in $\cB_0'$), more precisely: 
$$\forall g'\in\cB_0',\ \ \|g'\|_0 = 
\langle |g'|, 1_{\X} \rangle\ \ \ 
\mbox{and}\ \ \ \ \forall g'\in\cB_V',\ \ \|g'\|_V = 
\langle |g'|, V\rangle.$$
\begin{lem} \label{lem-D-F} 
Let $\delta\in(\delta_V(P),1)$. Then, there exist $N\in\N^*$ and $d\in(0,+\infty)$ such that for all $f'\in\cB_V'$ we have: $\|P^{*N}f'\|_V \leq \delta^N\|f'\|_V + d\|f'\|_0$.
\end{lem}
\begin{proof}{} 
Let $f'\in\cB_V'$ and $n\geq1$. Since $P^n$ is a nonnegative operator on $\cB_V$, so is its adjoint operator $P^{*n}$ on $\cB_V'$, and we have for all $f\in\cB_V$ such that $\|f\|_V\leq1$ (ie.~$|f| \leq V$): 
$$\big|\langle (P^{*})^n f', f \rangle\big| \leq \big\langle (P^{*})^n|f'|, |f| \big\rangle \leq \big\langle (P^{*})^n|f'|, V \big\rangle = \langle |f'|, P^nV \rangle.$$
By definition of $\delta_V(P)$ and from $\delta\in(\delta_V(P),1)$, there exist $N\in\N^*$ and $d\in(0,+\infty)$ such that $P^NV \leq \delta^N\, V + d\, 1_{\X}$. Thus   
\begin{eqnarray*}
\|(P^{*})^N f'\|_V &:=& \sup_{f\in{\cal B}_V,\|f\|_V\leq1} \big|\langle (P^{*})^N f', f \rangle \big| \\ 
&\leq& \langle |f'|, P^{N}V\rangle \\ 
&\leq& \delta^N\, \langle |f'|,V \rangle + d\, \langle |f'|,1_{\X}\rangle = \delta^N\|f'\|_V + d \|f'\|_0. 
\end{eqnarray*}
\end{proof}

\subsection{Comparison with Wu's work and further statements} \label{Sec_Wu}
Quasi-compactness of Markov kernels acting on $\cB_V$ is fully studied in \cite[Th.~3.11]{Wu04}. A first difference between Theorem~\ref{pro-qc-bis} and \cite[Th.~3.11]{Wu04} concerns their proofs. The proof of Theorem~\ref{pro-qc-bis}  is much more direct than in \cite{Wu04} since it uses Doeblin-Fortet inequalities. The next remarks show that Wu's  assumptions and conclusion are different from ours, but also complementary.  

\begin{itemize}
	\item $(\X,d)$ is assumed to be a Polish space in \cite{Wu04}, and Wu's topological assumptions on $P$ are the following ones: 
	
{\bf (A1')} {\it $\ \ P(x,dy)$ and $P_V(x,dy) := V(x)^{-1}V(y)P(x,dy)$ satisfy Hypothesis~(A1)}\footnote{The use of $P_V$ is crucial in \cite{Wu04}: indeed $P_V$ is a bounded operator on $\cB_0$ which has the same spectral properties as $P$ acting on $\cB_V$. The statement \cite[Th.~3.11]{Wu04} is then deduced from the study of the essential spectral radius of bounded kernels acting on $\cB_0$, see \cite[Th.~3.10]{Wu04}. Note that when $P$ satisfies Hypothesis~(A1), the same holds for $P_V$ whenever $PV^p/V^p$ is bounded on $\X$ for some $p\in(1,+\infty)$, see \cite{Wu04}.}

where Wu's hypothesis (A1) (introduced in \cite[p.~265]{Wu04}) uses sophisticated parameters $\beta_w(\cdot)$ and $\beta_\tau(\cdot)$ for measure of non-compactness of $P$. Our topological assumption, namely $P^\ell$ (for some $\ell\geq1$) is compact from $\cB_0$ to $\cB_V$, is more manageable and it only involves the kernel $P$ (not $P_V$). Furthermore, using duality ($(P^*)^\ell$ is compact from $\cB_0'$ to $\cB_V'$), our compactness assumption corresponds to one of the standard hypotheses of \cite{Hen93}.   

\item  The contraction-type condition in \cite[Th.~3.11]{Wu04} involves equivalent functions to $V$. When $V(x) \r +\infty$ as $d(x,x_0)\r+\infty$, it writes as follows (See $(a.i)\Leftrightarrow (a.iii)$ in \cite[Th.~3.11]{Wu04}): there exists an equivalent function $W$ (i.e.~$c^{-1}V\leq W \leq cV$) such that 
  $$r(W) := \limsup_{x\r\infty} \frac{(PW)(x)}{W(x)} <1.$$ 
  In practice, finding such a function $W$ is not easy, excepted of course when we directly have $PV \leq \delta V + d 1_{\X}$ with some $\delta\in(0,1)$ and $d>0$ (in this case $W=V$). Our contraction-type condition is: 
  $$\delta_V(P)  < 1.$$
That all the iterates of $P$ are involved in the definition (\ref{def-hat-rho}) of $\delta_V(P)$, and so in a bound of $r_{ess}(P)$, is quite natural from the spectral definition of $r_{ess}(P)$. Moreover, since the definition of $\delta_V(P)$ is only based on the function $V$ (not on equivalent functions), our contraction condition is more manageable than in \cite{Wu04}. 

\item Wu's conclusion \cite[(3.17)]{Wu04} states that $r_{ess}(P)$ is equal to the infimum of the quantities $r(W)$ over all the equivalent functions $W$. Using this formula to compute $r_{ess}(P)$ seems to be very difficult in practice (anyway such computations are not reported in Wu's examples). Finally, combining Theorem~\ref{pro-qc-bis} which gives the inequality $r_{ess}(P) \leq \delta_V(P)$, and Wu's result enables us to prove in Corollary~\ref{cor-Wu-egalite} that the expected equality $r_{ess}(P) = \delta_V(P)$ holds in many case.
\end{itemize}

\begin{cor} \label{cor-Wu-egalite} 
Assume that $\X$ is a Polish space, that $P$ satisfies Condition {\bf (WD)}, that $P^\ell$ is compact from $\cB_0$ to $\cB_V$ for some $\ell\geq 1$, that the topological assumptions {\bf (A1')} of \cite[Th.~3.11]{Wu04} are satisfied, and finally that 
$PV$ is bounded on each compact set of $\X$. Then: 
$$r_{ess}(P) = \delta_V(P).$$ 
\end{cor}
\begin{proof}{} From Theorem~\ref{pro-qc-bis} we know that $r_{ess}(P) \leq \delta_V(P)$. Let $r>r_{ess}(P)$. From \cite[Th.~3.11]{Wu04} there exists an  function $W$ equivalent to $V$ (i.e.~$c^{-1}V\leq W \leq cV$) such that $PW \leq r W + d1_{\X}$ (since $PW$ is bounded on compact sets). Iterating the last inequality shows that there exists $e>0$ such that: $\forall n\geq1,\ P^nW \leq r^n W + e\, 1_{\X}$. Thus we obtain $\forall n\geq1,\ P^nV \leq c^2r^n V + c\, e\, 1_{\X}$, so that we have for any $\rho > r$ and for $N$ sufficiently large: $P^NV \leq \rho^N V + c\, e\, 1_{\X}$. Therefore: $\delta_V(P) \leq \rho$. Since $r$ is arbitrarily close to $r_{ess}(P)$, so is $\rho$. This gives: $\delta_V(P) \leq r_{ess}(P)$. 
\end{proof}
\begin{itemize}
	\item In practice, {\bf (A1')} is deduced from the following conditions (see \cite[p.~265]{Wu04}): 

\indent {\bf (A2')} {\it $\ \ P$, $P_V$ are Feller and $P^{\ell}$, $P_V^{\ell}$ are strongly Feller (for some $\ell\geq 1$).}  

Recall that a nonnegative kernel $T(\cdot,dy)$ on $\X$ satisfying $\sup_{x\in\X} T(x,\X) <\infty$ is said to be Feller (respectively strongly Feller) if, for every bounded continuous (respectively measurable) function $f:\X\r\R$, the function $(Tf)(\cdot) := \int f(y)T(\cdot,dy)$ is continuous. In the general setting of Markov operators, it seems to be difficult to compare our compactness assumption with Hypothesis~{\bf (A1')}, and even with {\bf(A2')}. However, for absolutely continuous kernels, Hypothesis~{\bf(A2')} is stronger than our compactness assumption as explained below. 
\end{itemize}
{\bf Conditions} (K). 
{\it $(\X,d)$ is a separable metric space equipped with its Borel $\sigma$-algebra $\cX$. Every closed ball of $\X$ is compact. For some $x_0\in\X$ we have $\lim_{d(x,x_0)\r+\infty} V(x) = +\infty$. Finally there exist a positive measure $\eta$ on $(\X,\cX)$ and a measurable function $K : \X^2\r[0,+\infty)$ such that: }
\begin{equation} \label{gene-kernel-K-P}
\forall x\in\X, \quad P(x,dy) = K(x,y)\, d\eta(y). 
\end{equation}

\begin{lem} \label{lem-fel-comp}
Assume that Conditions~{\em (K)} hold. If $P^{\ell}$ is strongly Feller for some $\ell\geq1$, then $P^{2\ell}$ is compact from $\cB_0$ to $\cB_V$. 
\end{lem}
Although Lemma~\ref{lem-fel-comp} is a classical statement, we prove it in Appendix~\ref{ap-exist-pi} for completeness. Obviuously we deduce from Lemma~\ref{lem-fel-comp} that, if $P^m$ satisfies Conditions~(K) for some $m\geq1$ and if $P^{\ell}$ is strongly Feller for some $\ell\geq1$, then $P^{2\ell m}$  is compact from $\cB_0$ to $\cB_V$. This shows that the compactness assumption of Theorem~\ref{pro-qc-bis} is fulfilled in all the examples of \cite[Sect.~8]{Wu04}. Furthermore, Lemma~\ref{lem-fel-comp} and Theorem~\ref{pro-qc-bis} allows us to derive the following result. 
\begin{cor} \label{rem-compacite-B0-BV} 
 Assume that Conditions~{\em (K)} and {\bf (WD)} hold and that $P^{\ell}$ is strongly Feller for some $\ell\geq1$. Then $P$ is a power-bounded quasi-compact operator on $\cB_V$ with $$r_{ess}(P) \leq \delta_V(P).$$
\end{cor}
\begin{rem} \label{Rem_Noyau_Continu}
If $P$ is given by (\ref{gene-kernel-K-P}) with $K$ continuous in the first variable, then $P$ is strongly Feller. Indeed, for all $(x,x')\in\X^2$, we have: 
$$\big|(Pf)(x') - (Pf)(x) \big| \leq \int_\X \big|K(x',y) - K(x,y)\big|\, d\eta(y).$$
Since we have $K(\cdot,\cdot)\geq0$, $\int K(\cdot,y)d\eta(y) =1$, and $\lim_{x'\r x} K(x',y) = K(x,y)$, we deduce from Scheffé's theorem that $\lim_{x'\r x} \int_\X |K(x',y) - K(x,y)|\, d\eta(y) = 0$. This proves the desired statement. The previous argument even shows that $\{Pf,\, \|f\|_0\leq 1\}$ is equicontinuous. Observe that, when the last equicontinuity property holds and $V(x) = +\infty$ as $d(x,x_0)\r+\infty$, Ascoli's theorem is another way to prove the compactness of $P : \cB_0\r\cB_V$. 
\end{rem}

Let us introduce a last statement which gives a suitable sum up of the present discussion, in particular for analyzing ours models of Markov chains in Section~\ref{sub-sec-countable}. We know from \cite[p.~270]{Wu04} that if $P$ is  Feller and $P^{\ell}$ is strongly Feller for some $\ell\geq 1$, and if $V$ and $PV$ are continuous, then $P_V$ is Feller and ${P_V}^{\ell}$ is strongly Feller, and thus {\bf (A1')} holds true (see {\bf (A2')}). Therefore, using Corollary~\ref{cor-Wu-egalite} and Lemma~\ref{lem-fel-comp}, we obtain the following statement.
\begin{cor} \label{Cor_ref}
 Assume that $\X$ is a Polish space, that Conditions~{\em (K)} and {\em (\ref{cond-D})} hold, that $V$ and $PV$ are continuous, that $P$ is Feller and $P^{\ell}$ is strongly Feller for some $\ell \ge 1$. Then $P$ is a power-bounded quasi-compact operator on $\cB_V$ with $$r_{ess}(P) = \delta_V(P).$$
\end{cor}
%

\subsection{From quasi-compactness on $\cB_V$ to $V$-geometrical ergodicity} \label{sec-qc-V-geo}
Recall that a Markov chain $(X_n)_{n\in\N}$ with transition kernel  $P$ is $V$-geometrically ergodic if $P$ has an invariant probability measure $\pi$ such that
\begin{enumerate}[(VG1)]
	\item \label{VG1} $\pi(V)<\infty$
	\item \label{VG2} $\displaystyle\lim_{n \r \infty}\sup_{f\in{\cal B}_V, \|f\|_{V}\le 1}\| P^n f - \pi(f) \|_{V} =0$.  
\end{enumerate}
Let $\Pi$ denotes the rank-one projection defined on $\cB_V$ by: $\Pi f = \pi(f)1_{\X}$. Note that the condition (VG\ref{VG2}) is equivalent to the convergence to $0$ of $\|P^n - \Pi\|_{V}$, the operator norm associated with $\|\cdot\|_V$. Moreover, using $P^n - \Pi = (P - \Pi)^n$, it can be shown that the convergence is geometric, that is, there exists $\rho \in(0,1)$ and $c_\rho\in(0,+\infty)$ such that 
\begin{equation}
\label{strong-ergo} 
\|P^n - \Pi\|_V \leq c_\rho\, \rho^n. 
\end{equation}
Recall that the infimum bound of the positive real numbers $\rho$ such that (\ref{strong-ergo}) holds has been denoted by $\rho_V(P)$ and  called \textit{the convergence rate} of $P$ on $\cB_V$.  

In this subsection we propose a result which makes explicit the relationship between the quasi-compactness of $P$ and the $V$-geometric ergodicity of the Markov chain $(X_n)_{n \in\N}$ with transition kernel $P$. Moreover, we provide an explicit formula for $\rho_V(P)$ in terms of the spectral elements of $P$. A key element is the essential spectral radius $r_{ess}(P)$. For general quasi-compact Markov kernels on $\cB_V$, the result \cite[Th.~4.6]{Wu04} provides interesting additional material on peripheral eigen-elements. 

\begin{theo} \label{CNS-qc-Vgeo}
Let $P$ be a transition kernel which has an invariant probability measure $\pi$ such that $\pi(V)<\infty$. The two following assertions are equivalent:
\begin{enumerate}[(a)]
\item $P$ is $V$-geometrically ergodic.
\item $P$ is a power-bounded quasi-compact operator on $\cB_V$, for which $\lambda=1$ is a simple eigenvalue  (i.e.~$\ker(P-I) = \C\cdot 1_{\X}$) and the unique eigenvalue of modulus one.
\end{enumerate}
Under any of these conditions, we have $\rho_V(P) \ge r_{ess}(P)$. In fact, for ${r_0}\in(r_{ess}(P),1)$, denoting the set of all the eigenvalues $\lambda$ of $P$ such that ${r_0}\leq|\lambda|<1$ by $\cV_{r_0}$, we have:
\begin{itemize}
    \item either $\rho_V(P) \leq {r_0}$ when $\cV_{r_0}=\emptyset$,
    \item or $\rho_V(P) = \max\{|\lambda|,\, \lambda\in\cV_{r_0}\}$ when $\cV_{r_0}\neq\emptyset$.
\end{itemize}
Moreover, if  $\cV_{r_0}=\emptyset$ for all ${r_0}\in(r_{ess}(P),1)$, then $\rho_V(P) =r_{ess}(P)$.
\end{theo}
From Definition~\ref{def-q-c}, for any ${r_0}\in (r_{ess}(P),1)$, the set of all the eigenvalues of $\lambda$ of $P$ such that ${r_0}\leq|\lambda|\leq1$ is finite. 

\begin{rem}
The property that $P$ admits a spectral gap on $\cB_V$ in the recent paper \cite{KonMey11} corresponds here to the quasi-compactness of $P$ (which is a classical terminology in spectral theory). The spectral gap in \cite{KonMey11} corresponds to the value $1-\rho_V(P)$. Then, \cite[Prop.~1.1]{KonMey11}) is a reformulation of the equivalence of properties (a) and (b) in Theorem~\ref{CNS-qc-Vgeo} under $\psi$-irreducibility and aperiodicity assumptions (see also \cite[Lem.~2.1]{KonMey11}). The last statements in Theorem~\ref{CNS-qc-Vgeo} provide the value of the convergence rate $\rho_V(P)$ for $V$-geometrically ergodic Markov chains from the essential spectral radius $r_{ess}(P)$ and the (possible) eigenvalues $\lambda$ such that $r_{ess}(P) < |\lambda|<1$. 
\end{rem}

\noindent\begin{proof}{} 
Note that we have $\cB_V = \C\, 1_\X\oplus H_0$, with $H_0:=\{f\in\cB_V: \pi(f)=0\}$ (write $f=\pi(f)1_\X + (f-\pi(f)1_\X)$). Since $\pi(V)<\infty$, $\pi$ defines a bounded linear form on $\cB_V$, so that $H_0$ is a closed subspace of $\cB_V$. From the invariance of $\pi$, we obtain that $P(H_0)\subset H_0$. 

Now assume that $(a)$ is fulfilled. Then for any $\rho\in(\rho_V(P),1)$ we have from (\ref{strong-ergo}): 
$$\sup_{h\in H_0,\, \|h\|_V\leq1}\|P^nh\|_V = O(\rho^n).$$
It follows from Definition~\ref{def-q-c}  that $P$ is quasi-compact on $\cB_V$, with $r_{ess}(P)\leq \rho_V(P)$. The fact that $P$ is power-bounded on $\cB_V$ easily follows from $(a)$. 

Conversely, assume that $(b)$ holds and prove that Property~$(a)$, together with the claimed properties on $\rho_V(P)$,  are fulfilled. Since $P$ is Markov and power-bounded on $\cB_V$, we have $r(P)=1$. From Definition~\ref{def-q-c} 
and the assumption on the peripheral eigenvalues of $P$, we obtain for any ${r_0}\in(r_{ess}(P),1)$: 
\begin{equation} \label{dec-V_geo}
\cB_V = \C\, 1_\X\oplus \big(\oplus_{\lambda\in{\cal V}_{r_0}}\ker(P-\lambda I)^{p_\lambda}\big)\oplus H,
\end{equation}
where 
$H$ is a closed $P$-invariant subspace of $\cB_V$ such that $\sup_{h\in H,\, \|h\|_V\leq1}\|P^nh\|_V = O({r_0}^n)$. 
Let $f\in\cB_V$. Then we have 
\begin{equation} \label{dec-V_geo-bis} 
f-\pi(f)1_\X = \sum_{\lambda\in{\cal V}_{r_0}} f_\lambda + h,
\end{equation}
with $f_\lambda\in\ker(P-\lambda I)^{p_\lambda}$ and $h\in H$, and there exist some constants $c_\lambda$ and $c_H$ (independent of $f$) such that $\|f_\lambda\|_V\leq c_\lambda\|f\|_V$ and $\|h\|_V\leq c_H\|f\|_V$ (since the projections associated with the decomposition (\ref{dec-V_geo}) are continuous). 

When $\cV_{r_0}=\emptyset$, then (\ref{dec-V_geo-bis}) yields 
\begin{equation*} 
\|P^nf-\pi(f)1_\X\|_V = \|P^n\big(f-\pi(f)1_\X\big)\|_V \leq O({r_0}^n)\, \|f\|_V.
\end{equation*}
Thus Property~$(a)$ holds and $\rho_V(P) \leq {r_0}$. If 
$\cV_{r_0}=\emptyset$ for all ${r_0}\in(r_{ess}(P),1)$, then $\rho_V(P) \leq r_{ess}(P)$, so that $\rho_V(P) = r_{ess}(P)$ from the proof of $(a)\ \Rightarrow (b)$.

When $\cV_{r_0}\neq\emptyset$, define $\nu := \max\{|\lambda|,\, \lambda\in\cV_{r_0}\}$. We have for $\lambda\in\cV_{r_0}$ and $n\geq p_\lambda$ 
\begin{eqnarray*}
\|P^n f_\lambda\|_V = \|(P-\lambda I+\lambda I)^n f_\lambda\|_V 
&\leq& \sum_{k=0}^{p_\lambda-1} \binom{n}{k} |\lambda|^{n-k}\|(P-\lambda I)^k\|_V\, \|f_\lambda\|_V \\
&\leq& c_\lambda\bigg(\sum_{k=0}^{p_\lambda-1} \binom{n}{k} |\lambda|^{n-k}\|(P-\lambda I)^k\|_V\bigg)\|f\|_V.
\end{eqnarray*}
Then for each $k=0,\ldots,p_\lambda-1$, we have $\binom{n}{k}|\lambda|^{n-k} = O(n^k|\lambda|^n) \leq O(n^k\nu^n)$. Thus $\|P^n f_\lambda\|_V = O(\rho^n)\, \|f\|_V$ for any $\rho\in (\nu,1)$. From (\ref{dec-V_geo-bis}) and ${r_0}\leq \rho$, we obtain: 
$$
\|P^nf-\pi(f)1_\X\|_V 
\leq \sum_{\lambda\in{\cal V}_{r_0}} \|P^nf_\lambda\| + \|P^n h\|_V \leq \big(O(\rho^n) + O({r_0}^n)\big)\, \|f\|_V = O(\rho^n)\, \|f\|_V. \label{O-rho-r}
$$
Since $\rho\in (\nu,1)$ is arbitrary, this gives $\rho_V(P) \leq \nu$. Conversely, given any $\lambda\in\cV_{r_0}$ and $f\in\cB_V$ such that $Pf=\lambda f$, we have $\pi(f)=0$ (use the invariance of $\pi$), and from $P^nf=\lambda^n f$ and the definition of $\rho_V(P)$, we easily deduce that $|\lambda|\leq\rho_V(P)$. Thus $\nu \leq \rho_V(P)$. 
\end{proof}

The next lemma provides the existence of the $P$-invariant probability measure under the weak drift  Condition~(\ref{cond-D}). This statement (see e.g.~\cite{MeyTwe93} for similar results) will be of interest in our examples for the use of Theorem~\ref{CNS-qc-Vgeo}. For completeness, a proof is given in Appendix~\ref{ap-exist-pi}.
\begin{lem} \label{lem-exist-proba-inv}
Assume that $(\X,d)$ is a separable complete metric space and that $V : \X\r[1,+\infty)$ is continuous and such that the set $\{V\leq \alpha\}$ is compact for every $\alpha\in(0,+\infty)$. Under Condition~\emph{(\ref{cond-D})}, there exists a $P$-invariant probability measure $\pi$ such that $\pi(V)<\infty$. 
\end{lem}

\subsection{Study of characteristic functions} \label{subsec-taille-fct-propre}
As explained in Introduction, the next theorem plays an important role in our work. 
\begin{theo} \label{pro-tail-fct-propre}
Assume that the weak drift condition \emph{(\ref{cond-D})} holds true.
If $\lambda\in\C$ is such that $\delta \leq |\lambda| \leq 1$, with $\delta$ given in \emph{(\ref{cond-D})}, and if $f\in\cB_V\cap\ker(P-\lambda I)^p$ for some $p\in\N^*$, then there exists $c\in(0,+\infty)$ such that 
$$|f| \leq c\, V^{\frac{\ln|\lambda|}{\ln\delta}}\, (\ln V)^{\frac{p(p-1)}{2}}.$$  
\end{theo}
The proof of Theorem~\ref{pro-tail-fct-propre} is based on the following lemma.
\begin{lem} \label{lem-ineq-lem-tail-fct}
Let $\lambda\in\C$ be such that $\delta \leq |\lambda| \leq 1$. Then 
\begin{equation} \label{ineq-lem-tail-fct}
\forall f\in\cB_V,\ \exists c\in(0,+\infty),\ \forall x\in\X,\quad |\lambda|^{-n(x)} \big|(P^{n(x)}f)(x)\big|  \leq c\, V(x)^{\frac{\ln|\lambda|}{\ln\delta}}
\end{equation}
with, for any $x\in\X$, $n(x) := \big\lfloor \frac{-\ln V(x) }{\ln\delta}\big\rfloor$, where $\lfloor \cdot\rfloor$ denotes the integer part function. 
\end{lem}
\begin{proof}{}
First note that the iteration of (\ref{cond-D}) gives 
\begin{equation} \label{itération-cond-D}
\forall k\geq 1,\quad P^{kN}V \leq \delta^{kN}\, V + d\big(\sum_{j=0}^{k-1}\delta^{jN}\big)\, 1_{\X} \leq \delta^{kN}\, V +  
\frac{d}{1-\delta^N}\, 1_{\X}.
\end{equation}
Let $g\in\cB_V$ and $x\in\X$. Using (\ref{itération-cond-D}), the positivity of $P$ and $|g|\leq \|g\|_V\, V$, we obtain with $b := d/(1-\delta^N)$: 
\begin{equation} \label{val-pro}
\forall k\geq1,\quad |(P^{kN}g)(x)| \leq (P^{kN}|g|)(x) \leq \|g\|_V\, (P^{kN}V)(x) \leq  \|g\|_V\big(\delta^{kN}V(x) + b\big).
\end{equation}
The previous inequality is also fulfilled with $k=0$. Next, let $f\in\cB_V$ and $n\in\N$. Writing $n = kN+r$, with $k\in\N$ and $r\in\{0,1,\ldots,N-1\}$, and applying (\ref{val-pro}) to $g:=P^rf$, we obtain with $\xi:=\max_{0\leq\ell\leq N-1}\|P^\ell f\|_V$ (use $P^nf = P^{kN}(P^rf)$): 
\begin{equation} \label{val-pro-bis}
\big|(P^nf)(x)\big| \leq \xi\big[\delta^{kN}V(x) + b\big] \leq 
\xi\big[\delta^{-r}\big( \delta^{n}V(x) + b\big)\big] \leq \xi\,\delta^{-N}\big(\delta^{n}V(x) + b\big).
\end{equation}
Using the inequality $$-\frac{\ln V(x)}{\ln\delta} -1 \leq n(x) \leq -\frac{\ln V(x)}{\ln\delta}$$ 
and the fact that  $\ln\delta \leq \ln|\lambda|\leq 0$, Inequality (\ref{val-pro-bis}) with $n:=n(x)$ gives: 
\begin{eqnarray*}
|\lambda|^{-n(x)} \big|(P^{n(x)}f)(x)\big| &\leq& \xi\, \delta^{-N}\bigg(\big(\delta|\lambda|^{-1}\big)^{n(x)}\, V(x) + b\, |\lambda|^{-n(x)} \bigg) \\
& &= \xi\, \delta^{-N}\bigg(e^{n(x)(\ln\delta - \ln|\lambda|)}\, e^{\ln V(x)} + b\, e^{-n(x)\ln|\lambda|} \bigg) \\
&\leq&  \xi\, \delta^{-N}\bigg(e^{(\frac{\ln V(x) }{\ln\delta}+1)\, (\ln|\lambda|-\ln\delta)}\, e^{\ln V(x)} + 
b\, e^{\frac{\ln V(x)}{\ln\delta}\ln|\lambda|} \bigg)  \\
&& = \xi\, \delta^{-N}\bigg(e^{\frac{\ln|\lambda|}{\ln\delta}\ln V(x)}\, e^{\ln|\lambda|-\ln\delta}\,  + b\, 
V(x)^{\frac{\ln|\lambda|}{\ln\delta}}\bigg) \\
&& =\xi\, \delta^{-N} \big(e^{\ln|\lambda|-\ln\delta}+b\big)\, V(x)^{\frac{\ln|\lambda|}{\ln\delta}}.
\end{eqnarray*}
This gives the desired conclusion with $c=\xi\, \delta^{-N} (e^{\ln|\lambda|-\ln\delta}+b)$. 
\end{proof}
\begin{proof}{ of Theorem~\ref{pro-tail-fct-propre}}
 If $f\in\cB_V\cap\ker(P-\lambda I)$, then $|\lambda|^{-n(x)} |(P^{n(x)}f)(x)| = |f(x)|$, so that (\ref{ineq-lem-tail-fct}) gives the expected conclusion when $p=1$. Next, let us  proceed by induction. Assume that the conclusion of Theorem~\ref{pro-tail-fct-propre} holds for some $p\geq 1$. Let $f\in\cB_V\cap\ker(P-\lambda I)^{p+1}$. We can write 
\begin{equation} \label{inter-tail-1}
P^nf = (P-\lambda I+\lambda I)^n f = \lambda^{n}\, f + \sum_{k=1}^{\min(n,p)} \binom{n}{k} \lambda^{n-k}\,  (P-\lambda I)^k f.
\end{equation}
For $k\in\{1,\ldots,p\}$, we have $f_k := (P-\lambda I)^k f \in\ker(P-\lambda I)^{p+1-k} \subset \ker(P-\lambda I)^{p}$, thus we have from the induction hypothesis : 
\begin{equation} \label{inter-tail-2}
\exists c'\in(0,+\infty),\ \forall k\in\{1,\ldots,p\}, \ \forall x\in\X, \quad |f_k(x)| \leq c'\, V(x)^{\frac{\ln|\lambda|}{\ln\delta}}\, (\ln V(x))^{\frac{p(p-1)}{2}}.
\end{equation}
Now, we obtain from (\ref{inter-tail-1}) (with $n:=n(x)$), (\ref{inter-tail-2}) and Lemma~\ref{lem-ineq-lem-tail-fct} that for all $x\in\X$: 
\begin{eqnarray*}
|f(x)| &\leq& |\lambda|^{-n(x)} \big|(P^{n(x)}f)(x)\big| + c'\, V(x)^{\frac{\ln|\lambda|}{\ln\delta}}\, (\ln V(x))^{\frac{p(p-1)}{2}}\, |\lambda|^{-\min(n,p)} \sum_{k=1}^{\min(n,p)} \binom{n(x)}{k} \\
&\leq&  c\, V(x)^{\frac{\ln|\lambda|}{\ln\delta}} + c_1\, V(x)^{\frac{\ln|\lambda|}{\ln\delta}}\, (\ln V(x))^{\frac{p(p-1)}{2}}\, n(x)^p \\
&\leq& c_2 V(x)^{\frac{\ln|\lambda|}{\ln\delta}}\, (\ln V(x))^{\frac{p(p-1)}{2} + p}
\end{eqnarray*}
with some constants $c_1,c_2\in(0,+\infty)$ independent of $x$. Since $p(p-1)/2 + p = p(p+1)/2$, this gives the expected result. 
\end{proof}

To conclude this section, notice that the $V$-geometrical ergodicity clearly implies Condition~(\ref{cond-D}). However Condition~(\ref{cond-D}) is not sufficient for $P$ to be $V$-geometrically ergodic, even if $P$ is assumed to be compact from $\cB_0$ from $\cB_V$. In fact, the previous statements provide the following procedure to check the $V$-geometric ergodicity of $P$ and to compute an upper bound for its convergence rate $\rho_V(P)$. Let $P$ be a transition kernel with an invariant probability measure $\pi$ such that $\pi(V)<\infty$. Theorem~\ref{pro-qc-bis} shows that, if $P$ (or some iterate) is compact from $\cB_0$ into $\cB_V$ and satisfies the weak drift  condition~(\ref{cond-D}), then $P$ is quasi-compact on $\cB_V$ and $r_{ess}(P)\le \delta_V(P)$. Next Theorem~\ref{CNS-qc-Vgeo} ensures that the $V$-geometric ergodicity of $P$ can be deduced from quasi-compactness  provided that the following properties are satisfied : 
\begin{enumerate}[(i)]
\item   $\lambda=1$ is a simple eigenvalue of $P$ on $\cB_V$, namely $\ker(P-I) = \C\cdot 1_\X$;
\item $\lambda=1$ is the unique eigenvalue of $P$ of modulus one on $\cB_V$. 
\end{enumerate} 
Finally Theorem~\ref{pro-tail-fct-propre} can be useful to check (i)-(ii),  and in a more general way to investigate the sets $\cV_{r_0}$ of eigenvalues of $P$ given in Theorem~\ref{CNS-qc-Vgeo} in order to obtain an upper bound for the convergence rate $\rho_V(P)$. This procedure is applied in the next section. 
%

\section{Applications to discrete Markov chains} \label{sub-sec-countable}
In this section, we are concerned with  discrete Markov chains. For the sake of simplicity, we assume that $\X:=\N$ throughout the section. Let $P$ be a Markov kernel on $\N$. The main focus is on the estimation of the essential spectral radius $r_{ess}(P)$ from Condition~(\ref{cond-D}): a general statement is derived from Corollary~\ref{rem-compacite-B0-BV} in Subsection~\ref{sub-gene-qc}, and applications to random walks (RW) with bounded state-dependent increments are presented in Subsection~\ref{ex-reflecting-nonhom}. 

For irreducible and aperiodic discrete Markov chains, criteria for the $V$-geometrical ergodicity are well-known from the literature using, either the equivalence between geometric ergodicity and $V$-geometric ergodicity of $\N$-valued Markov chains \cite[Prop.~2.4]{HorSpi92}, or the strong drift Condition~(\ref{inequality-drift}) with a small set $S$ \cite{MeyTwe93}. In Subsection~\ref{sub-cond-I-A}, we just explain as an alternative way how the quasi-compactness combined with irreducibility and aperiodicity conditions provide the $V$-geometrical ergodicity. Finally the procedure mentioned at the end of the previous section (see (i))-(ii)) is applied to compute the convergence rate of some random walks (see Example~\ref{ex-speksma-1} and Subsection~\ref{sub-basic-ex-revis}). 
Such computations are not reported in Wu's work, excepted for the examples \cite[Ex.~8.3-8.4]{Wu04} corresponding to the specific case $P(0,0):=1-q$ in Subsection~3.4. Wu obtained the rate $\rho_V(P)$ for these two examples as special instances of discrete reflected random walks. These processes, also called Lindley's random walks, are investigated in Subsection~5.2, in which we obtain the rate of convergence with explicit constant for general discrete Lindley's random walks.

\subsection{Quasi-compactness of discrete Markov chains} \label{sub-gene-qc}
Let $P=(P(i,j))_{i,j\in\N^2}$ be a Markov kernel on $\N$. The function $V:\N\r[1,+\infty)$ is assumed to satisfy 
$$\lim_nV(n) = +\infty \quad \text{ and } \quad \sup_{n\in\N}\frac{(PV)(n)}{V(n)}<\infty.$$ 
\begin{cor} \label{cor-qc-bis} 
The two following conditions are equivalent: 
\begin{enumerate}[(a)]
	\item Condition~\emph{(\ref{cond-D})} holds with $V$; 
	\item $\displaystyle L := \inf_{N\geq 1}(\ell_N)^{\frac{1}{N}}< 1$ where 
	$\ell_N := \limsup_{n\r+\infty}(P^NV)(n)/V(n)$.
\end{enumerate}
In this case, $P$ is power-bounded and quasi-compact on $\cB_V$ with 
$$r_{ess}(P) = \delta_V(P) = L. $$
\end{cor}
\begin{proof}{}
That $P$ is power-bounded and quasi-compact on $\cB_V$ with $r_{ess}(P) = \delta_V(P)$ under (\ref{cond-D}) follows  from Corollary~\ref{Cor_ref} since $P$ is strongly Feller in the discrete state space case. 

Let us prove the equivalence \emph{(a)}$\Leftrightarrow $\emph{(b)}, as well as the equality $\delta_V(P) = L$. 
First, Condition~(\ref{cond-D}) clearly gives  $\ell_N \leq \delta^N <1$ (with $\delta$ in (\ref{cond-D})), thus $L\leq \delta_V(P)$ by definition of $\delta_V(P)$. Conversely, assume that $L<1$: there exists $N\geq1$ such that $\ell_N <1$. Let $\delta$ be such that $\ell_N<\delta^N<1$. Then there exists $n_0\in\N$ such that: $\forall n>n_0,\ (P^NV)(n)/V(n) \leq \delta^N$. Hence 
$$P^NV\leq \delta^N V + d,\quad \text{with}\ d:= \max_{0\leq i \leq n_0} \frac{(P^NV)(i)}{V(i)}.$$
This proves (\ref{cond-D}), and $\delta_V(P) \leq (\ell_N)^{1/N}$ since $\delta$ is arbitrary close to $(\ell_N)^{1/N}$. In fact, the last argument shows that $\delta_V(P) \leq (\ell_N)^{1/N}$ provided that $\ell_N <1$. From definition of $L$, there exists a sequence $(N_k)_{k\geq0}$ such that $L=\lim_k(\ell_{N_k})^{1/N_k}$. Thus we have $\delta_V(P) \leq (\ell_{N_k})^{1/N_k}$ for $k$  large enough. Thus $\delta_V(P) \leq L$. 
\end{proof} 

In the next subsections, Corollary~\ref{cor-qc-bis} is applied to random walks on $\N$ with the following special sequence $ V_\gamma:=(\gamma^n)_{n\in\N}$ for some $\gamma\in(1,+\infty)$. The associated weighted-supremum  space $\cB_\gamma \equiv \cB_{V_\gamma}$ is defined by: 
\begin{equation} \label{def-Bgamma-discret}
\cB_{\gamma} := \big\{(f(n))_{n\in\N}\in\C^{\N} : \sup_{n\in\N}\gamma^{-n}|f(n)| < \infty \big\}.
\end{equation}

\subsection{Quasi-compactness of RW with bounded state-dependent increments} \label{ex-reflecting-nonhom}  
Let us fix $b\in\N^*$, and assume that the kernel $P$ on $\X:=\N$ satisfies the following conditions:   
\begin{gather}
\forall i\in\{0,\ldots,b-1\},\quad \sum_{j\ge 0} P(i,j)=1; \nonumber \\
\forall i\ge b, \forall j\in\N, \quad P(i,j) = 
\begin{cases}
 0 & \text{ if } |i-j|>b \\
 a_{j-i}(i) & \text{ if } |i-j|\le b \end{cases} \label{Def_NHRW}
\end{gather} 
where $(a_{-b}(i),\ldots,a_b(i))\in[0,1]^{2b+1}$ satisfies $\sum_{k=-b}^{b} a_k(i)=1$ for all $i\ge b$. This kind of kernels arises, for instance, from time-discretization of Markovian queueing models (see a basic example in Remark~\ref{rk-mrw-boucle}).

\begin{pro} \label{pro-non-hom}
Assume that for every $k\in\Z$ such that  $|k|\le b$
\begin{subequations}
\begin{equation} \label{ak(n)-lim}
 \lim_n a_k(n)=a_k\in[0,1],
\end{equation}
and that $\gamma\in(1,+\infty)$ is such that 
\begin{gather} 
  \phi(\gamma) := \sum_{k=-b}^{b} a_{k}\, \gamma^{k} <1 \label{non-hom-cont}\\
  \forall i\in\{0,\ldots,b-1\},\quad \sum_{j\ge 0} P(i,j)\gamma^j < \infty. \label{non-hom-moment} 
\end{gather}
\end{subequations}
Then $P$ is power-bounded and quasi-compact on $\cB_{\gamma}$ with 
$$r_{ess}(P) \leq \phi(\gamma).$$
\end{pro}
\begin{proof}{}
Set $\phi_n(\gamma) := \sum_{k=-b}^{b} a_{k}(n)\, \gamma^{k}$. We have $(PV_\gamma)(n) = \phi_n(\gamma) V_\gamma(n)$ for each $n\geq b$, hence we obtain from (\ref{ak(n)-lim})
$$\limsup_{n}\frac{(PV_\gamma)(n)}{V_\gamma(n)} \leq  \phi(\gamma).$$
The conclusion of Proposition~\ref{pro-non-hom} then follows from Corollary~\ref{cor-qc-bis} using (\ref{non-hom-cont})-(\ref{non-hom-moment}). 
\end{proof}

Assume that $a_0\neq1$. Let $\phi^{(k)}$ be the $k$-th derivative of $\phi$. It is checked in  Appendix~\ref{proof_prop1_NHRW} that there exists $1\leq\ell\leq 2b$ such that 
\begin{equation} \label{def-ell-RRW} 
\forall k\in\{1,\ldots,\ell-1\},\quad \phi^{(k)}(1) = 0
\quad \text{and} \quad \phi^{(\ell)}(1) \neq 0, 
\end{equation}
according that the first condition is removed when $\ell=1$. Since $\phi(1)=1$, 
a sufficient condition for (\ref{non-hom-cont}) to hold for some $\gamma\in(1,+\infty)$ is that 
\begin{equation} \label{CS_non-hom-cont}
\phi^{(\ell)}(1) < 0.	
\end{equation}
\begin{ex}[State-dependent birth-and-death Markov chains] \label{ex-bin-non-hom} 
When $b:=1$ in (\ref{Def_NHRW}), we obtain the standard class of state-dependent birth-and-death Markov chains. Namely,  the stochastic kernel $P$ is defined by:
$$\forall n\ge 1,\ P(n,n-1) := p_n,\quad P(n,n) := r_n,\quad P(n,n+1) := q_n,$$
where the real numbers $(p_n,r_n,q_n)\in[0,1]^2$ and $p_n+r_n+q_n=1$. Assume that the following limits exist:  
$$\lim_n p_n := p \in(0,1], \quad \lim_n r_n := r \in[0,1),\quad \lim_n q_n :=q.$$
If $\gamma\in(1,+\infty)$ is such that 
$$\phi(\gamma) := \frac{p}{\gamma} + r + q\gamma <1 \quad \text{and}\quad \sum_{n\geq0} P(0,n) \gamma^{n} < \infty,$$ 
then it follows from Proposition~\ref{pro-non-hom} that $P$ is power-bounded and quasi-compact on $\cB_{\gamma}$ with 
$$r_{ess}(P) \leq \frac{p}{\gamma} + r + q\gamma.$$ 
The conditions $\gamma >1$ and $p/\gamma + r + q\gamma <1$ are equivalent to the following ones (use $r=1-p-q$ for $(i)$): 
\begin{center}
\begin{tabular}{ll}
  $(i)$ either $p>q>0$ and $1<\gamma < p/q$; & $(ii)$ or $q=0$ and $\gamma>1$. 
\end{tabular}
\end{center}
\begin{itemize}
	\item When $p>q>0$ and $1<\gamma < p/q$: if $\sum_{n\geq0} P(0,n) \gamma^{n} < \infty$, then $P$ is power-bounded and quasi-compact on $\cB_{\gamma}$ with $r_{ess}(P) \leq \phi(\gamma)$. Set $\widehat\gamma := \sqrt{p/q}$. Then 
\begin{equation} \label{delta-het-gamma}
\min_{\gamma>1} \phi(\gamma) = \phi(\widehat\gamma) = r+2\sqrt{pq}\ \in(r^2,1).
\end{equation}
	Consequently, if $\sum_{n\geq0} P(0,n) (\widehat\gamma)^n < \infty$, then the previous conclusions holds for $\gamma:=\widehat\gamma$, with essential spectral radius on $\cB_{\widehat\gamma}$ satisfying 
$$r_{ess}(P) \leq r+2\sqrt{pq}.$$
  \item When $q:=0$ and $\gamma>1$: if $\sum_{n\geq0} P(0,n) \gamma^{n} < \infty$, then $P$ is power-bounded and quasi-compact on $\cB_{\gamma}$ with 
  $$r_{ess}(P) \leq \phi(\gamma)=p/\gamma + r.$$
  Such a case is illustrated by the next example. 

\end{itemize}
\end{ex} 

\begin{ex}[Simulation of a Poisson distribution with parameter one] \label{ex-poisson} 
The Markov kernel $P$ on $\X:=\N$ defined by 
\begin{gather*}
P(0,0)=P(0,1)=\frac{1}{2} \\
\forall n\ge 1,\ P(n,n-1) := \frac{1}{2},\quad P(n,n) := \frac{n}{2(n+1)},\quad P(n,n+1) := \frac{1}{2(n+1)}. \end{gather*} 
arises from a Hastings-Metropolis sampler of a Poisson distribution. 
We have $p=r=1/2$ and $q=0$ with the notations of Example~\ref{ex-bin-non-hom}. Hence, for each $\gamma\in(1,+\infty)$, $P$ is power-bounded and quasi-compact on $\cB_{\gamma}$ and 
$$r_{ess}(P) \leq 1/2 + 1/(2\gamma).$$ 
\end{ex}
\begin{rem}[Random walks with i.d. bounded increments] \label{HRW} 
Consider the case when the increments $a_k(n)$ do not depend on the state n, that is when 
the kernel $P$ is 
\begin{gather*}
\forall i\in\{0,\ldots,b-1\},\quad \sum_{j\ge 0} P(i,j)=1; \quad 
\forall i\ge b, \forall j\in\N, \quad P(i,j) = \left\{ \begin{array}{lcl}
a_{j-i}  & \text{if} & |i-j|\le b  \\
0 & \text{if}& |i-j|> b
\end{array}\right.
\end{gather*} 
where $(a_{-b},\ldots,a_b)\in[0,1]^{2b+1}$ and $\sum_{k=-b}^{b} a_k=1$. Obviously the statements of Example~\ref{ex-bin-non-hom} apply but some additional facts can be deduced for such Markov chains. 
First note that 
\begin{equation} \label{itere-V-gamma}
\forall \gamma\in(1,+\infty),\ \forall N\geq 1,\ \forall n\geq Nb,\quad (P^NV_{\gamma})(n) = \phi(\gamma)^N\, V_{\gamma}(n).
\end{equation}
Consequently, under the assumptions (\ref{non-hom-moment}) and $\phi(\gamma)<1$ where $\phi(\cdot)$ is given by (\ref{non-hom-cont}), we obtain from Corollary~\ref{cor-qc-bis} that Condition~\emph{(\ref{cond-D})} is fulfilled with $V_{\gamma}$ and 
\begin{equation} \label{delta-RWW}
r_{ess}(P)=\delta_{V_{\gamma}}(P) = \phi(\gamma).
\end{equation}
Moreover, it is shown in Appendix~\ref{app-equi-WD} that, under the assumptions $a_0\neq 1$ and (\ref{non-hom-moment}), Condition~\emph{(\ref{cond-D})} holds true with $V_{\gamma_0}$ for some $\gamma_0\in(1,\gamma]$ if and only if $\phi^{(\ell)}(0)<1$ (see (\ref{def-ell-RRW})-(\ref{CS_non-hom-cont})). Finally, for the  birth-and-death Markov chains, that is when $b:=1$, the convergence rate can be computed (see Subsection~\ref{sub-basic-ex-revis}). 
\end{rem}
\subsection{$V$-geometrical ergodicity for discrete Markov chains} \label{sub-cond-I-A}

Let $P=(P(i,j))_{i,j\in\N^2}$ be a Markov kernel on $\N$. The following irreducibility and aperiodicity conditions for discrete Markov chains are well-known. For any $(i,j)\in\N^2$, define
$$\cR_{i,j} := \big\{n\geq1 : P^{n}(i,j)>0\big\}.$$
 The Markov kernel $P$ is said to be irreducible if  
\begin{equation} \label{cond-ireduc}
\forall (i,j)\in\N,\quad \cR_{i,j}\neq\emptyset, \tag{\textbf{${\cal I}$}}
\end{equation} 
and to be aperiodic if 
\begin{equation} \label{cond-aperiod}
\exists\, i\in\N,\quad \cR_{i,i} - \cR_{i,i} := \{n-m,\, (m,n)\in\cR_{i,i}\times\cR_{i,i}\} = \Z. \tag{\textbf{${\cal A}$}}
\end{equation}
Since $\cR_{i,i}$ is stable under addition from the Chapman-Kolmogorov equation, the subgroup of $\Z$ generated by $\cR_{i,i}$ coincides with $\cR_{i,i} - \cR_{i,i}$. Hence the aperiodicity Condition~(\ref{cond-aperiod}) is equivalent to the usual one: the largest element $d=d(i)\in\N^*$ such that $\cR_{i,i}\subset d\cdot\N^*$ (i.e. the g.c.d. of $\cR_{i,i}$), called the period of $i$, is equal to $1$. If $P$ is irreducible then each state $j\in\N$ has the same period. 
\begin{cor} \label{cor-geo-I-A} 
Under Condition~\emph{(\ref{cond-D})}, $P$ has an invariant probability measure $\pi$ such that $\pi(V)<\infty$. If the additional Conditions~\emph{(\ref{cond-ireduc})}-\emph{(\ref{cond-aperiod})} hold true, then $P$ is $V$-geometrically ergodic. 
\end{cor}
The previous statement is well-known. It can be derived from quasi-compactness (note that the first assertion follows from Lemma~\ref{lem-exist-proba-inv}): apply Corollary~\ref{cor-qc-bis} and Theorem~\ref{CNS-qc-Vgeo} with Conditions~\emph{(\ref{cond-ireduc})}-\emph{(\ref{cond-aperiod})} (see Appendix~\ref{app-compl-cor3} for completeness).  

Under the additional Conditions~\emph{(\ref{cond-ireduc})}-\emph{(\ref{cond-aperiod})}, all the statements of Subsection~\ref{ex-reflecting-nonhom} can be completed in order to find again the $V_\gamma$-geometric ergodicity. For instance, in Example~\ref{ex-bin-non-hom}, the $V_\gamma$-geometrical  ergodicity holds when  $P(0,0)\in(0,1)$ and $p_n,q_n>0$ for all $n\geq1$.\footnote{Note that, if $P(0,0)=0$, then the period $d(0)$ of $i:=0$ may be equal to 2. For instance this fact holds when $P(0,1)=1$, and $\lambda=-1$ is then an eigenvalue of $P$: $P$ is quasi-compact on $\cB_{\widehat\gamma}$, but it is not $V_{\widehat\gamma}$-geometrically ergodic.} 
In Example~\ref{ex-poisson}, Conditions~\emph{(\ref{cond-ireduc})}-\emph{(\ref{cond-aperiod})} are automatically fulfilled so that $P$ is $V_\gamma$-geometrically ergodic without additional assumptions.

Note that the irreducibility condition is not necessary for $P$ to be $V$-geometrically ergodic: in this case the use of Theorem~\ref{CNS-qc-Vgeo} (via Corollary~\ref{cor-qc-bis}) is of interest to obtain the $V$-geometric ergodicity as illustrated in the following simple example. 
\begin{ex} [An instance of binary RW] \label{ex-birth-and-death} 
Assume that 
$$\forall n\ge 1,\ P(n,n-1) := p\in(0,1],\quad P(n,n) := r =1-p.$$
Under the assumptions  $\gamma>1$ and $\sum_{n\geq0} P(0,n) \gamma^{n} < \infty$, we know from (\ref{delta-RWW}) that $r_{ess}(P) = p/\gamma + r$. Note that Condition~\emph{(\ref{cond-ireduc})} is not automatically fulfilled in this instance. Anyway, without additional assumptions, $P$ is $V_\gamma$-geometrically ergodic. Indeed Theorem~\ref{CNS-qc-Vgeo} applies. First the equation $Pf=f$ leads to: $\forall n\geq1, f(n)=f(n-1)$, so that $f$ is constant. Hence $1$ is a simple eigenvalue of $P$. Second, given $\lambda\in\C$, $|\lambda|=1$, $\lambda\neq1$, any solution of $Pf=\lambda f$ is of the form: $f=(f(0){z_\lambda}^n)_{n\in\N}$ with $z_\lambda:=p/(\lambda-r)$. From $|\lambda-r| > 1-r=p$, we obtain $|z_\lambda| <1$, so that the equality $(Pf)(0)=\lambda f(0)$, namely $\lambda f(0) = f(0) \sum_{n\geq0}P(0,n){z_\lambda}^n$ is only possible when $f(0)=0$. Hence $1$ is the only eigenvalue of modulus one. 
\end{ex}

Finally recall that, as it was outlined at the end of Section~\ref{sec-mino}, quasi-compactness is especially of interest for bounding the convergence rate of $P$. Example~\ref{ex-speksma-1} below is a first simple illustration of this fact. Other applications to birth-and-death Markov chains are proposed in the next Subsection. 
\begin{ex}[An instance of RW with unbounded increments] \label{ex-speksma-1} 
Let us point out that Corollary~\ref{cor-qc-bis} and Theorem~\ref{CNS-qc-Vgeo} may be also useful for random walks on $X:=\N$ with unbounded increments. For instance, let $P$ be defined by \cite{MalSpi95}
$$\forall n\geq 1,\ P(0,n) := q_n,\quad \forall n\geq 1,\ P(n,0) := p,\ P(n,n+1) := q = 1-p,$$
with $p\in(0,1)$ and $q_n\in[0,1]$ such that $\sum_{n\geq1}q_n=1$. For $\gamma\in(1,+\infty)$ and $V_\gamma:=(\gamma^n)_{n\in\N}$, we have: $\forall n\geq1,\ (PV_\gamma)(n) = p +q\gamma^{n+1} = (p/\gamma^n + q\gamma)V_\gamma(n)$. Thus, if $\gamma\in(1,1/q)$ and $\sum_{n\geq1}q_n\gamma^n < \infty$, then Condition~\emph{(\ref{cond-D})} holds with $V_\gamma$ and we have $\delta_{V_\gamma}(P) \leq q\gamma$. Therefore, under the previous conditions, if follows from Corollary~\ref{cor-qc-bis} that $P$ is power-bounded, quasi-compact on $\cB_{\gamma}$  and  
$$r_{ess}(P) \leq q\gamma.$$ 

No additional assumptions are required to obtain the $V_\gamma$-geometric ergodicity: $P$ is $V_\gamma$-geometrically ergodic provided that $\gamma\in(1,1/q)$ and $\sum_{n\geq1}q_n\gamma^n < \infty$. Moreover the convergence rate $\rho_{V_{\gamma}}(P)$ of $P$ on $\cB_\gamma$ satisfies:  
\begin{equation} \label{spect-gap-speksma-1} 
\rho_{V_\gamma}(P) \leq \max(q\gamma,p).
\end{equation}
Proof of (\ref{spect-gap-speksma-1}). \ 
Theorem~\ref{CNS-qc-Vgeo} is applied with any $r_0>\max(q\gamma,p)$. Let $\lambda\in\C$ be such that $\max(q\gamma,p) < |\lambda| \leq1$, and let $f\in\cB_\gamma$, $f\neq0$, be such that $Pf=\lambda f$. We obtain $f(n) = (\lambda/q)f(n-1) - pf(0)/q$ for any $n\geq 2$, so that
\begin{eqnarray*}
\forall n\geq2,\quad f(n) &=& f(1)\, \left(\frac{\lambda}{q}\right)^{n-1}  - \frac{pf(0)}{q}\, \bigg(\frac{1-(\frac{\lambda}{q})^{n-1}}{1-\frac{\lambda}{q}}\bigg) \\
&=& \left(\frac{\lambda}{q}\right)^{n-1}\left(f(1)- \frac{pf(0)}{\lambda-q}\right) + \frac{pf(0)}{\lambda-q}.
\end{eqnarray*}
Since $f\in\cB_\gamma$  and $|\lambda|/q > \gamma$, we obtain $f(1)= pf(0)/(\lambda-q)$, and consequently: $\forall n\geq1,\ f(n)= pf(0)/(\lambda-q)$. Next the equality $\lambda f(0) = (Pf)(0) = \sum_{n\geq1} q_nf(n)$ gives: $\lambda f(0) = pf(0)/(\lambda-q)$ since $\sum_{n\geq1}q_n=1$. We have $f(0)\neq0$ since we look for a solution $f\neq0$. Thus $\lambda$ satisfies $\lambda^2 - q\lambda - p=0$, namely: $\lambda = 1$ or $\lambda =-p$. The case $\lambda =-p$ has not to be considered from assumption. If $\lambda = 1$, we have $f(n)=f(0)$ for each $n\in\N$, so that $1$ is a simple eigenvalue.  We have proved that $1$ is a simple eigenvalue of $P$ on $\cB_\gamma$ and that $\lambda=1$ is the only eigenvalue of $P$ on $\cB_\gamma$ such that $\max(q\gamma,p) < |\lambda| \leq1$. Then Theorem~\ref{CNS-qc-Vgeo} gives the estimate (\ref{spect-gap-speksma-1}) of the convergence rate. Note that $p$ cannot be dropped in (\ref{spect-gap-speksma-1}) since $\lambda=-p$ is an eigenvalue of $P$ on $\cB_\gamma$ with corresponding eigenvector (up to a multiplicative constant) $f_p:=(1,-p,-p,\dots)$. 
\end{ex}
%
\subsection{Study of the convergence rate for the birth-and-death Markov chains} \label{sub-basic-ex-revis} 
%
We consider real numbers $p,q,r\in[0,1]$ such that $p+r+q=1$, $p>q>0$, and we assume that $P$ is defined on $\X:=\N$ by 
\begin{equation} \label{hyp-ref-rand}
\begin{array}{c}
 \forall n\ge 1,\ P(n,n-1) :=p,\quad P(n,n) := r\in[0,1),\quad P(n,n+1) := q, \\[1mm]
 \displaystyle P(0,0)\in(0,1),\quad \sum_{n\geq0} P(0,n) \,({\widehat\gamma})^{n} < \infty \text{ where  }\widehat\gamma:= \sqrt{\frac{p}{q}}\in(1,+\infty).
 \end{array}
\end{equation} 
Let $V_{\widehat\gamma}:=(\widehat\gamma^n)_{n\in\N}$. The weighted-supremum space $\cB_{\widehat\gamma}:=\cB_{V_{\widehat\gamma}}$ associated to $V_{\widehat\gamma}$ is defined in (\ref{def-Bgamma-discret}). Note that Conditions~\emph{(\ref{cond-ireduc})}-\emph{(\ref{cond-aperiod})} hold true. We know from  Corollary~\ref{cor-geo-I-A} that $P$ is $V_{\widehat\gamma}$-geometrically ergodic and from (\ref{delta-het-gamma}) (\ref{delta-RWW}) that 
$$r_{ess}(P) = \delta_{V_{\widehat\gamma}}(P) = r+2\sqrt{pq}.$$

As illustrated afterwards, thanks to Theorem~\ref{CNS-qc-Vgeo} and Lemma~\ref{lem-ref-rand-vp} below, the last estimate on the essential spectral radius $r_{ess}(P)$ is relevant to compute the convergence rate $\rho_{V_{\widehat\gamma}}(P)$.   
\begin{lem} \label{lem-ref-rand-vp}
Assume that Conditions~(\ref{hyp-ref-rand}) hold true. If $f$ is a nontrivial eigenvector in $\cB_{\widehat\gamma}$ associated with a complex eigenvalue $\lambda$ of $P$ such that $r+2\sqrt{pq} < |\lambda| \leq 1$ then 
\begin{equation} \label{f-z1}
\exists \alpha_1\in\C\setminus\{0\},\ \forall n\geq 0,\ f(n) = \alpha_1\, {z_\lambda}^n,
\end{equation}
with $z_\lambda$ satisfying the following conditions: 
\begin{subequations}
\begin{gather} 
 |z_\lambda| <  \widehat\gamma, \label{mod-z1} \\
 q{z_\lambda}^2 + (r - \lambda) z_\lambda + p = 0, \label{ex1-cond1-bax} \\
 \sum_{n\geq0} P(0,n){z_\lambda}^n = \lambda. \label{ex1-cond0-bax-gene} 
\end{gather}
\end{subequations}
\end{lem}
\begin{proof}{}
Let $\lambda\in\C$ be such that $r+2\sqrt{pq} < |\lambda| \leq 1$. Let $f\in\cB_{\widehat\gamma}$, $f\ne 0$ such that $Pf=\lambda f$ so that 
\begin{equation} \label{ex1-cond1-bax-bis}
\forall n\geq 1,\quad \lambda f(n) = pf(n-1) + rf(n) + qf(n+1). 
\end{equation}
Let us denote by $z_\lambda,z_\lambda'$ the two complex solutions  of the characteristic equation 
$$qz^2 + (r - \lambda) z + p = 0.$$ 
Observe that $z_\lambda z_\lambda'=p/q= \widehat\gamma^2$. Recall that the solutions of (\ref{ex1-cond1-bax-bis}) are of the form, either $f(n) = \alpha_1 {z_\lambda}^n + \alpha_2 {z_\lambda'}^n$ if $z_\lambda\neq z_\lambda'$, or $f(n) = \alpha_1 {z_\lambda}^n + \alpha_2 n {z_\lambda}^n$ if $z_\lambda=z_\lambda'$, with $\alpha_1,\alpha_2\in\C$. 

We have $|z_\lambda|\neq |z_\lambda'|$. Indeed, Theorem~\ref{pro-tail-fct-propre} applied with $p:=1$ and $\delta := r+2\sqrt{pq}$ implies  that $|f|\leq c {V_{\widehat\gamma}}^\tau$ with $\tau := \ln|\lambda|/\ln\delta\in(0,1)$ and some   constant $c$. Consequently we have $|\alpha_1 {z_\lambda}^n + \alpha_2 {z_\lambda'}^n| \leq c\, {\widehat\gamma}^{\tau n}$ in case $z_\lambda\neq z_\lambda'$, and $|\alpha_1 {z_\lambda}^n + \alpha_2 n{z_\lambda}^n| \leq c\,{{\widehat\gamma}}^{\tau n}$ in case  $z_\lambda=z_\lambda'$. If $|z_\lambda| = |z_\lambda'|$, then we would have $|z_\lambda| = |z_\lambda'| = \widehat\gamma$, but the two previous inequalities then easily imply that $\alpha_1=\alpha_2=0$, that is $f=0$.

From $|z_\lambda| \neq |z_\lambda'|$, we can suppose that (for instance) $|z_\lambda| <  \widehat\gamma$ and $|z_\lambda'| >  \widehat\gamma$. Since $f$, $({z_\lambda}^n)_{n\in\N}$ are in $\cB_{\widehat\gamma}$ and $({z_\lambda'}^n)_n$ is not in $\cB_{\widehat\gamma}$, we obtain: $\forall n\geq 0,\ f(n) = \alpha_1 {z_\lambda}^n$. Since $f\neq0$ (i.e.~$\alpha_1\neq0$), the equation $(Pf)(0) = \lambda f(0)$ implies that $z_\lambda$ must satisfy (\ref{ex1-cond0-bax-gene}). 
\end{proof}
\begin{pro} \label{BD} 
In addition to Conditions~(\ref{hyp-ref-rand}), the boundary transition probabilities are assumed to satisfy, for some $a\in(0,1)$: 
$$P(0,0):=a,\ \ P(0,1):=1-a.$$
Then $P$ is $V_{\widehat\gamma}$-geometrically ergodic. Furthermore, defining  $a_0 := 1-q-\sqrt{pq}$, the convergence rate $\rho_{V_{\widehat\gamma}}(P)$ of $P$ is given by:    
\begin{itemize}
	\item when $a\in[a_0,1)$: 
  \begin{equation} \label{a0-2}
	\rho_{V_{\widehat\gamma}}(P) = r+2\sqrt{pq}\, ; 
  \end{equation} 
	\item when $a\in(0,a_0]$: 
\begin{enumerate}[(a)]
	\item in case $\, 2p \leq \big(1-q+\sqrt{pq}\big)^2$: 
\begin{equation}\label{a0-0}
\rho_{V_{\widehat\gamma}}(P) = r+2\sqrt{pq}\, ; 
\end{equation} 
  \item in case $\, 2p > \big(1-q+\sqrt{pq}\big)^2$, setting $a_1 := p - \sqrt{pq}  - \sqrt{r\big(r+2\sqrt{pq}\big)}$:
\begin{subequations}
\begin{eqnarray}
& &  \rho_{V_{\widehat\gamma}}(P) = \left|a + \frac{p(1-a)}{a-1+q}\right| \ \ \text{ when } a\in(0,a_1] \label{a0-1} \\
& & \rho_{V_{\widehat\gamma}}(P) =  r+2\sqrt{pq}  \ \, \quad\qquad\text{ when } a\in[a_1,a_0). \label{a0-3}
\end{eqnarray}
\end{subequations}
\end{enumerate}
\end{itemize}
\end{pro}

When $r:=0$ in the previous proposition, we have   $a_0 = a_1 = p-\sqrt{pq} = (p-q)/(1+\sqrt{q/p})$, and it can be easily checked that $2p > (1-q+\sqrt{pq})^2$. The properties (\ref{a0-2}) (\ref{a0-1}) (\ref{a0-3}) then rewrite as: 
\begin{subequations}
\begin{eqnarray}
& &  \rho_{V_{\widehat\gamma}}(P) = \frac{pq+(a-p)^2}{|a-p|} \text{ when } a\in(0,a_0]\label{a0-1-r0} \\
& &  \rho_{V_{\widehat\gamma}}(P) = 2\sqrt{pq} \ \, \quad\qquad\text{ when } a\in(a_0,1). \label{a0-2-r0}
\end{eqnarray}
\end{subequations}
Using Kendall's theorem, the properties (\ref{a0-1-r0})-(\ref{a0-2-r0}) have been proved for $a< p$ in \cite{RobTwe99} and \cite[Ex.~8.4]{Bax05}. For $a\geq p$, (\ref{a0-2-r0}) can be derived from \cite{LunTwe96}  using the fact that $P$ is stochastically monotone. Our method gives a unified and simple proof of (\ref{a0-1-r0})-(\ref{a0-2-r0}), and encompasses the case $r\neq0$.  

\noindent\begin{proof}{ of Proposition~\ref{BD}}
By elimination, given some $\lambda\in\C$, a necessary and sufficient condition for the two following equations: 
\begin{subequations}
\begin{gather} 
 q z^2 + (r - \lambda) z + p = 0, \label{ex1-cond1-bax-ter} \\
 a+(1-a)z = \lambda. \label{ex1-cond0-bax-gene-bis}
\end{gather}
\end{subequations}
to have a common solution $z\in\C$ is that 
\begin{eqnarray} 
0 = \left|\begin{array}{ccc}
1-a & a-\lambda & 0 \\
0 & 1-a & a-\lambda \\
q & r-\lambda & p \\
\end{array}
\right| 
&=& (1-\lambda) 
\left|\begin{array}{ccc}
1-a & 1 & 0 \\
0 & 1 & a-\lambda \\
q & 1 & p \\
\end{array}
\right| \nonumber \\ 
&=& (1-\lambda)\big[(\lambda-a)(1-a-q) +p(1-a) \big] \label{eq-lambda-glob}.
\end{eqnarray}

Assume that $a \neq 1-q$. Then $\lambda=1$ is a solution of (\ref{eq-lambda-glob}) and the other solution of (\ref{eq-lambda-glob}), say $\lambda(a)$, and the associated complex number in (\ref{ex1-cond0-bax-gene-bis}), say $z(a)$, are given by: 
\begin{equation} \label{lambda=}
\lambda(a) :=  a + \frac{p(1-a)}{a-1+q}\in\R \quad \text{ and } \quad  z(a) := \frac{p}{a+q-1}\in\R.
\end{equation}

Now, let $\lambda\in\C$ be such that $r+2\sqrt{pq} < |\lambda| < 1$, and assume that there exists $f\in\cB_{\widehat\gamma}$, $f\neq0$, such that $Pf=\lambda f$. Then Lemma~\ref{lem-ref-rand-vp} gives $f:= ({z_\lambda}^n)_{n\geq0}$ (up to a multiplicative constant), with $z_\lambda\in\C$ satisfying $|z_\lambda| <  \widehat\gamma$ and Equations (\ref{ex1-cond1-bax-ter})-(\ref{ex1-cond0-bax-gene-bis}). Thus we have  $\lambda=\lambda(a)$ and $z_\lambda=z(a)$, with 
$\lambda(a)$ and $z(a)$ given by (\ref{lambda=}). Conversely, we have $Pf_a=\lambda(a) f_a$ with $f_a=(z(a)^n)_{n\ge 0}$ since, by definition, $z(a)$ satisfies the equations (\ref{ex1-cond1-bax-ter})-(\ref{ex1-cond0-bax-gene-bis}) associated with $\lambda=\lambda(a)$. Now we must find the values $a\in(0,1)$ for which we have $r+2\sqrt{pq} < |\lambda(a)| < 1$ and $|z(a)|\leq \widehat\gamma$. This is the relevant question since Theorem~\ref{CNS-qc-Vgeo} gives the following properties:   
\begin{enumerate}[(i)]
	\item if $r+2\sqrt{pq} < |\lambda(a)| < 1$ and $|z(a)| <  \widehat\gamma$, then we have $\rho_{V_{\widehat\gamma}}(P) = |\lambda(a)|$ since $\lambda(a)$ is the only eigenvalue $\lambda$ of $P$ on $\cB_{\widehat\gamma}$ such that $r+2\sqrt{pq} < |\lambda| < 1$ (apply Theorem~\ref{CNS-qc-Vgeo} with any $r_0$ such that $r+2\sqrt{pq}<r_0<|\lambda(a)|$), 
	\item if $\lambda(a)$ or $z(a)$ do not satisfy the previous conditions, then we have $\rho_{V_{\widehat\gamma}}(P) = r+2\sqrt{pq}$ since there is no eigenvalue $\lambda$ of $P$ on $\cB_{\widehat\gamma}$ such that $r+2\sqrt{pq} < |\lambda| < 1$ (apply Theorem~\ref{CNS-qc-Vgeo} with any $r_0$ such that $r+2\sqrt{pq} < r_0 <1$). 

\end{enumerate}

First, observe that 
\begin{equation} \label{machin}
|z(a)|\leq \widehat\gamma  \ \Leftrightarrow\ |a-1+q|\geq \sqrt{pq}. 
\end{equation}
Hence, if $a\in(a_0,1)$ (recall that $a_0 := 1-q-\sqrt{pq}$), then $|z(a)| >\widehat\gamma$. Then (ii) gives (\ref{a0-2}). 

Second consider the case $a\in(0,a_0]$. Then we have $|z(a)| \leq \widehat\gamma$, and we have to study $\lambda(a)$.  Note that $\lambda'(a) = 1-pq/(a-1+q)^2$, so that the function $a\mapsto \lambda(a)$ is increasing on $(-\infty,a_0]$ from $-\infty$ to $\lambda(a_0)=r-2\sqrt{pq}$.  
Thus
$$\forall a\in(0,a_0],\quad \lambda(a)\leq r-2\sqrt{pq} < r+2\sqrt{pq}.$$ 
and the equation $\lambda(a) = -(r+2\sqrt{pq})$ has a unique solution $a_1\in(-\infty,a_0)$. 
For the continuation, it suffices to have in mind that $a_1<a_0$ and $\lambda(a_1) = -(r+2\sqrt{pq})$, that $\lambda(0) = p/(q-1)\in[-1,0)$ and finally that 
$$\lambda(0) - \lambda(a_1) = p/(q-1) +r+2\sqrt{pq} = \frac{(q-\sqrt{pq}-1)^2 - 2p}{1-q}.$$
When $2p \leq (1-q+\sqrt{pq})^2$, (\ref{a0-0}) follows from (ii). Indeed $|\lambda(a)| < r+2\sqrt{pq}$ since 
$$\forall a\in(0,a_0],\quad -(r+2\sqrt{pq}) = \lambda(a_1) \leq \lambda(0) < \lambda(a) < r+2\sqrt{pq}.$$
When $2p > (1-q+\sqrt{pq})^2$, we have $a_1\in(0,a_0]$ and:  
\begin{itemize}
	\item if $a\in(0,a_1)$, then (\ref{a0-1}) follows from (i). Indeed $r+2\sqrt{pq} < |\lambda(a)| < 1$ since 
	$$\forall a\in(0,a_1],\quad -1 \leq \lambda(0) < \lambda(a) < \lambda(a_1) = -(r+2\sqrt{pq})\, ;$$ 
	\item if $a\in[a_1,a_0]$, then (\ref{a0-3}) follows (ii). Indeed  $|\lambda(a)| < r+2\sqrt{pq}$ since 
	$$-(r+2\sqrt{pq}) =\lambda(a_1) \leq \lambda(a) < r+2\sqrt{pq}.$$ 
\end{itemize}

It remains to study the special case $a = 1-q$. Then $\lambda=1$ is the only solution of (\ref{eq-lambda-glob}). Again let $\lambda\in\C$ be such that $r+2\sqrt{pq} < |\lambda| < 1$, and let $f\in\cB_{\widehat\gamma}$, $f\neq0$, such that $Pf=\lambda f$. Then Lemma~\ref{lem-ref-rand-vp} gives $f:= ({z_\lambda}^n)_{n\geq0}$, with $z_\lambda\in\C$ satisfying Equations~(\ref{ex1-cond1-bax-ter})-(\ref{ex1-cond0-bax-gene-bis}), thus Equation~(\ref{eq-lambda-glob}).  Consequently there is no eigenvalue of $P$ such that $r+2\sqrt{pq} < |\lambda| < 1$.  Theorem~\ref{CNS-qc-Vgeo} applied with any $r_0\in(r+2\sqrt{pq},1)$ then gives $\rho_{V_{\widehat\gamma}}(P) = r+2\sqrt{pq}$. 
\end{proof}

\begin{rem} \label{rk-mrw-boucle}
Let us consider the time-discretised M/M/1 queue obtained using the uniformization technique \cite[Section~4.1]{HorSpi92}. The arrival and service rates are denoted by $\beta>0$ and $\mu>0$ respectively. For $0<h < 1/(\beta+\mu)$, the kernel $P_h$ is defined by $P_h = I+hQ$ where $Q$ is the generator of the continuous time birth-and-death process $(X_t)_{t\ge 0}$ of the number of customers in a M/M/1 queue, so that  
\begin{gather*}
	P_h(0,0):=1-\beta h,\quad P_h(0,1):= \beta h ; \\
	 \forall n\ge 1, \quad   P_h(n,n-1)=\mu h, \quad P_h(n,n):=1-h(\beta+\mu),\quad  P_h(n,n+1)=\beta h.
\end{gather*}
Assume that $\beta/\mu<1$ which is the ergodicity condition of the M/M/1 queue. 
Note that $\mu h, \beta h, 1-h(\beta+\mu),1-\beta h$ stand for $p,q,r,a$ with the notations of Proposition~\ref{BD} and that $p >q$, $a=P_h(0,0)\in(0,1)$ and $a+q-1=0$. Therefore, for any $0<h < 1/(\beta+\mu)$, $P_h$ is  $V_{\widehat\gamma}$-geometrically ergodic with $\widehat\gamma=\sqrt{\mu/\beta}$  and 
\[ 
\rho_{V_{\widehat\gamma}}(P_h)= 1-h(\sqrt{\mu}-\sqrt{\beta})^2.
\]
Note that $\rho_{V_{\widehat\gamma}}(P_h)$ is decreasing as $h$ growth to $1/(\beta+\mu)$. The minimum (not attained) $2\sqrt{pq}=2h\sqrt{\mu\beta}=2\sqrt{\mu\beta}/(\mu+\beta)$ would be obtained as $h:=1/(\beta+\mu)$ but in this case $r:=0$ and we retrieve a binary random walk (and its convergence rate $\rho_{V_{\widehat\gamma}}=2\sqrt{\mu\beta}$) corresponding to the embedded Markov chain associated with the birth-and-death Markov process. This last value was found to be the $\L^1$-convergence rate in \cite{LunTwe96} using the monotone structure of the Markov chain and to be the $\L^2$-convergence rate in \cite{RobTwe01} from the equality of the two rates using the reversibility of the model. Finally, note that exponential bounds for $(X_t)_{t\ge 0}$ are easily derived using that its semi-group $(P_t)_{t\ge 0}$ satisfies $P_t=\exp(Qt)= \sum_{k=0}^{\infty} {P_h}^k \exp(-t/h) (t/h)^k/k! $ and that $P_h$ has the same invariant probability measure than $(X_t)_{t\ge 0}$ for any $h<1/(\beta+\mu)$. Indeed, we obtain that for any $0<\rho < (\sqrt{\mu}-\sqrt{\beta})^2$: 
\begin{equation*}
 \forall f\in \cB_{\widehat\gamma}, \quad  \|P_tf -\pi(f)\|_{V_{\widehat\gamma}}= \sup_{n\in\N}\frac{|(P_tf)(n) -\pi(f)|}{V_{\widehat\gamma}(n)} =  O\big( \exp(-t\rho)\big).
\end{equation*}
Note that $(\sqrt{\mu}-\sqrt{\beta})^2$ is know to be the $\L^2(\pi)$-spectral gap for such a process (see e.g. \cite{Kar00}). 
\end{rem}
%

\section{$V$-geometrical ergodicity of  iterated function systems} \label{sect-strong-ergo}

In this section we assume that $(\X,d)$ is a Polish space equipped with its borel $\sigma$-algebra $\cX$. Let $(\V,\cV)$ be a measurable space. Let us first recall the definition of an iterated function system (IFS) of Lipschitz maps (see \cite{DiaFre99,Duf97}).  
\begin{defi}[IFS of Lipschitz maps] \label{defi-LIFS}
Let $(\vartheta_n)_{n\geq 1}$ be a sequence of $\V$-valued i.i.d.~random variables, with common distribution denoted by $\nu$. Let $X_0$ be a $\X$-valued r.v.~which is assumed to be independent of the sequence $(\vartheta_n)_{n\geq 1}$. Finally, let $F : (\V\times \X,\cV\otimes\cX)\r (\X,\cX)$ be jointly measurable and Lipschitz continuous with respect to the second variable. The associated iterated function system (IFS) is the sequence of random variables $(X_n)_{n\in\N}$ which, given $X_0$, is recursively defined by: 
\begin{equation} \label{ifs-formule} 
\forall n\geq 1,\ \ \ X_n := F(\vartheta_n, X_{n-1}).
\end{equation}
\end{defi}
Clearly $(X_n)_{n\in\N}$ is a Markov chain, with transition kernel $P$:  
\begin{equation} \label{def-Q} 
\forall x\in\X,\ \forall A\in\cX,\quad P(x,A) = \E[1_A\big(F(\vartheta_1,x)\big)] = \int_{\V} 1_A\big(F(v,x)\big)\, d\nu(v).  
\end{equation}

Let $x_0\in\X$ be fixed. For any $b\in [0,+\infty)$, we set 
$$\forall x\in\X,\ p(x):=1+\, d(x,x_0) \quad \text{ and } \quad V_b(x) := p(x)^b.$$ 
We simply denote by $(\cB_b,|\cdot|_b)$ the weighted-supremum Banach space $\cB_{V_b}$ associated with $V_b(\cdot)$, that is 
\begin{equation} \label{def-Bb} 
	\cB_b := \left\{f : \X\r\C \text{ measurable such that } |f|_b  := \sup_{x\in \X} \frac{|f(x)|}{p(x)^b} < \infty \right\}.
\end{equation}
If $\psi : (\X,d)\r(\X,d)$ is a Lipschitz continuous function, we define 
\begin{equation} \label{L-psi}
L(\psi) := \sup\left\{\frac{d\big(\psi(x),\psi(y)\big)}{d(x,y)},\ (x,y)\in \X^2,\ x\neq y\right\}.
\end{equation}
Let ${a}\in[1,+\infty)$. We denote by $\cL_{a}$ the following space:  
\begin{equation} \label{def-La} 
\cL_{a} := \left\{f : \X\r\C \ : \  m_{a}(f) :=  \sup\bigg\{\frac{|f(x)-f(y)|}{d(x,y)\, (p(x) + p(y))^{a-1}},\ (x,y)\in \X^2,\  
x\neq y\bigg\}\, <\, \infty \right\}.  
\end{equation}
Such Lipschitz-weighted spaces have been introduced in \cite{LeP83} to obtain quasi-compactness of Lipschitz kernels, see also \cite{MilRau89,Duf97,Ben98,HenHer01}. 

Note that, for $f\in\cL_{a}$, we have for all $x\in\X$: $|f(x)|\leq |f(x_0)| + 2^{a-1}\,m_{a}(f)\,p(x)^{a}$. Thus: 
$$\forall f\in\cL_{a}, \quad  |f|_{a} := \sup_{x\in \X}\ \frac{|f(x)|}{p(x)^{a}} < \infty \quad \text{ and } \quad \cL_a \subset \cB_a .$$
$\cL_a$ equipped with the norm $\|f\|_a := m_a(f) + |f|_a$ is a Banach space. 

In Subsection~\ref{subsect-strong-ergo} we give standard contraction/moment conditions, called $(\cC_{a})$, for $P$ to have a geometric rate of convergence on $\cL_a$. In Subsection~\ref{sub-rate-BV} the passage to the $V_a$-geometric ergodicity is investigated.  As already mentioned in Introduction, this section is close to \cite[Sect.~7.2-8]{Wu04}. Under Conditions~$(\cC_{a})$, an IFS satisfies the contractive property \cite[(7.2)]{Wu04} with respect to the Wassertein distance introduced by Wu. But his topological hypothesis on $P$ is replaced by our more general compactness assumption on $P^\ell : \cB_0\r\cB_a$ (for some $\ell\geq1$). The bounds obtained in Corollary~\ref{pro-D-K1} on $r_{ess}(P)$ and $\rho_{V_a}(P)$ (for $P$ acting on $\cB_a$) are the same as in \cite{Wu04}. The rates of convergence with explicit constants,  obtained in Subsection~\ref{rem-rand-walk-N} for discrete Lindley's random walks and in Subsection~\ref{ex1-auto} for autoregressive models, are not reported in \cite{Wu04}. 

\subsection{Basic inequalities for IFS} \label{subsect-strong-ergo}
 
For all $x\in \X$, $v\in \V$ and $(v_1,\ldots,v_n)\in\V^n$  ($n\in\N^*$), define:  
\begin{subequations}
\begin{eqnarray}
& & F_v x : = F(v,x) \quad \text{ and } \quad L(v) := L(F_v) \label{Fv} \\
& & F_{v_n:v_1}:=F_{v_n}\circ \cdots\circ F_{v_1} \quad \text{ and } \quad L(v_n:v_1) := L(F_{v_n:v_1}). \label{Fn-F1}
\end{eqnarray}
\end{subequations}
By hypothesis we have $L(v)<\infty$, and so $L(v_n:v_1) <\infty$. Note that, for each $a\ge 1$, the limit 
$$\hat\kappa_a := \lim_{n\r+\infty}\E\left[L(F_{\vartheta_n:\vartheta_1})^{{a}}\right]^{\frac{1}{na}}$$
exists in $[0,+\infty]$, since the sequence $(\E[L(\vartheta_n:\vartheta_1)^{{a}}])_{n\in\N^*}$ is submultiplicative. 
Let us consider the following classical moment/contraction conditions: 

\noindent {\bf Conditions~$(\cC_{a})$}. {\it For some ${a}\in[1,+\infty)$: }
\begin{subequations}
\begin{eqnarray}
& & \E\left[d(F_{\vartheta_1}x_0,x_0)^{{a}}\right] < \infty \label{ite-moment-mu1-mu2}  \\
& &\widehat{\kappa}_a  < 1. \label{de-hatkappa-a}
\end{eqnarray}
\end{subequations}
\begin{pro}[see \cite{Duf97,Ben98}] \label{pi}
Under Conditions $(\cC_{a})$, there exists a unique $P$-invariant distribution, denoted by $\pi$, on $(\X,\cX)$, and we have $\pi(d(x_0,\cdot)^a)<\infty$. 
\end{pro}
\begin{pro}[\cite{Duf97}] \label{ergo} 
Under Conditions $(\cC_{a})$, the transition kernel $P$ continuously acts on $\cL_{a}$, and for any $\kappa\in(\widehat{\kappa}_a,1)$, there exists positive constants $c\equiv c_\kappa$ and $c'\equiv c'_\kappa$ such that:  
\begin{subequations}
\begin{eqnarray}
&\ & 
\forall n\geq1,\ \forall f\in \cL_{a},\quad |P^nf-\pi(f)1_\X|_{a} \leq c\, \kappa^n\, m_a(f) \label{vit-lip} \\
&\ & 
\forall n\geq1,\ \forall f\in \cL_{a},\quad \|P^nf-\pi(f)1_\X\|_{a}\leq  c'\, \kappa^n\, \|f\|_a.  \label{vit-lip-strong-ergo}
\end{eqnarray}
\end{subequations}
In particular, if $\kappa_1 := \E[L(\vartheta_1)^{a}]^{\frac{1}{a}} < 1$, then  
\begin{equation} \label{vit-lip-bis}
\forall f\in \cL_{a},\ \forall n\geq1,\quad |P^nf-\pi(f)1_\X|_{a} \leq  c_1 \, \kappa_1^n\, m_a(f), 
\end{equation}
where the  constant $c_1$ is defined by $c_1 := \xi^{(a-1)/a}\, \|\pi\|_1 \big(1 + \|\pi\|_a\big)^{a-1}$, with 
$$\xi := \sup_{n\geq 1}\sup_{x\in\X} \frac{(P^nV_a)(x)}{V_a(x)} < \infty\quad \text{ and } \quad  \|\pi\|_b := \left(\int_\X p(y)^b\, d\pi(y)\right)^{1/b}\ \ \text{ for } b:=1,a.$$ 
\end{pro}
Properties~(\ref{vit-lip})--(\ref{vit-lip-strong-ergo}) and (\ref{vit-lip-bis}) can be derived from the results of \cite[Chapter 6]{Duf97}. For convenience they are proved in Appendix~\ref{BB}. 
Note that the properties (\ref{vit-lip}) and (\ref{vit-lip-bis}) do not provide the $V_a$-geometric ergodicity since they are only established for $f\in\cL_a$. Indeed, in general the spaces $\cL_a$ and $\cB_a$ do not coincide, even for countable Markov chains. 
\begin{rem} \label{rem-ergo-bis} 
Under the conditions $(\cC_{a})$ and $\E[L(\vartheta_1)^{a}]^{\frac{1}{a}} < 1$, the proof of \emph{(\ref{vit-lip-bis})} (see Appendix~\ref{BB}) gives the following bound for the constant $\xi$ of Proposition~\ref{ergo}. Given any real number $\delta$ satisfying $\E[L(\vartheta_1)^{{a}}] < \delta < 1$, choose $r$ such that 
$$d(x,x_0) >  r\ \Rightarrow\ \E\left[\left(\frac{1 + L(\vartheta_1)\, d(x,x_0) + d(F_{\vartheta_1}x_0,x_0)}{1+d(x,x_0)}\right)^a\right] \leq \delta.$$
Then we have, with $\xi_1 := \E\big[\big(\max(1,L(\vartheta_1)) + d(F_{\vartheta_1}x_0,x_0)\, \big)^a\big]$, 
$$\xi \leq 1+\frac{\xi_1(1+r)^a}{1-\delta} .$$ 
\end{rem}
\begin{rem} \label{cor-f-Pf}
From Proposition~\ref{ergo} we deduce the following fact. Assume that Conditions $(\cC_{a})$ holds and that $f: \X\r\R$ is such that $P^\ell f\in\cL_a$ for some $\ell\in\N^*$. Then, for any $\kappa\in(\widehat{\kappa}_a,1)$, we have 
\begin{equation} \label{res-cor-f-Pf}
\forall n\geq \ell,\quad  |P^nf-\pi(f)1_\X|_{a} = \big|P^{n-\ell}(P^\ell f-\pi(f)1_\X)\big|_{a} \leq c\, \kappa^{n-\ell}\, m_a(P^\ell f), 
\end{equation}
where $c\equiv c_\kappa$ is the constant of Proposition~\ref{ergo}. If $\kappa=\kappa_1 := \E[L(\vartheta_1)^{a}]^{1/a} < 1$, then $c=c_1$. 
 Although this is not directly connected with the $V_a$-geometric ergodicity, such estimate may be of interest. This is illustrated in Corollary~\ref{cor-auto-1}.
\end{rem}
\begin{ex}[A simple example] \label{rem-IFS-unif}  Let $(X_n)_{n\in\N}$ be the real-valued IFS 
$$X_0\in\R, \qquad \forall n\ge 1, \quad X_n := \vartheta_nX_{n-1},$$
associated with $F_v x:=v x$ and with a sequence $(\vartheta_n)_{n\geq1}$ of i.i.d. random variables assumed to be independent of $X_0$. This kind of multiplicative Markov models are popular in finance. Let us assume that the $\vartheta_n$'s have a uniform probability distribution on $[0,1]$. The transition kernel $P(x,dy)$ of $(X_n)_{n\geq 0}$ is the uniform distribution on $[0,x]$ if $x>0$ (resp.~on $[x,0]$ if $x<0$). The Dirac distribution $\delta_0$ at $0$ is clearly $P$-invariant. Finally, setting $x_0:=0$ and $d(x,y):=|x-y|$, we have: $\forall a\in[1,+\infty)$
$$\E[d(F_{\vartheta_1}0,0)^{{a}}] = 0 \quad \text{and}\quad \kappa_1:=\E[L(\vartheta_1)^{a}]^{\frac{1}{a}} = \E[{\vartheta_1}^{a}]^{\frac{1}{a}} = \left(\frac{1}{a+1}\right)^{\frac{1}{a}}.$$
Consequently Inequality~(\ref{vit-lip-bis}) is valid. If $a:=1$, then the constant $c_1$ in (\ref{vit-lip-bis}) is equal to  
$1$ since we have $p(x)=1+|x|$ and $\pi=\delta_0$. 
\end{ex}

\begin{ex}[Autoregressive model] \label{rem-IFS-3}  
In this example, we prove that Inequality~(\ref{vit-lip-bis}) is fulfilled with the (optimal) value $\kappa_1:=\widehat{\kappa}_a$. Let $(X_n)_{n\in\N}$ be the real-valued IFS 
$$X_0\in\R, \qquad \forall n\ge 1, \quad X_n := \alpha\, X_{n-1} + \vartheta_n,$$
associated with $F_v x:=\alpha x + v$ where $\alpha\in(-1,1)$ is fixed and with a sequence $(\vartheta_n)_{n\geq1}$ of  centered random variables. This is the so-called autoregressive model of order 1 with an arbitrary centered noise. We take $d(x,y):=|x-y|$ and $x_0:=0$, so that  
 $L(v_n:v_1) = |\alpha|^n$ and $|F_v0| = |v|$. Then, we have $\widehat{\kappa}_a = |\alpha|$ for all $a\in(1,+\infty)$. In particular we have $\widehat{\kappa}_a = \E[L(\vartheta_1)^{a}]^{1/a} < 1$. Therefore, Conditions~$(\cC_{a})$ reduce to the moment condition $\E[|\vartheta_1|^{a}] < \infty$, and under this condition, Inequality~(\ref{vit-lip-bis}) holds  for $\kappa_1 := |\alpha|$. 
 
 Next, let us check that $|\alpha|$ is the minimal value of $\kappa$ (and $\kappa_1$) such that  Inequality~(\ref{vit-lip}) (and Inequality~(\ref{vit-lip-bis})) is valid. Since $\vartheta_1$ is centered, we have: $\forall x\in\X,\ \ \E[X_1\mid X_0=x]=\E[\alpha x+\vartheta_1] = \alpha x$. In other words, we have $P\phi = \alpha\phi$ where $\phi$ is the identity function on $\R$ (i.e.~$\phi(x):=x$). Note that $\phi$ is in $\cL_a$ for every $a\geq1$, and that $\pi(\phi)=0$ using $P\phi = \alpha\phi$. Consequently, under the condition~$\E[|\vartheta_1|^{a}] < \infty$, Inequality~(\ref{vit-lip}) holds for $f:=\phi$, and we can deduce from $P^n\phi = \alpha^n\phi$ that Inequality~(\ref{vit-lip}) cannot be valid with some $\kappa < |\alpha|$. 
\end{ex}
\subsection{Application to discrete Lindley's random walk} \label{rem-rand-walk-N}  
Recall that a Lindley random walk is defined on $\X:=[0,+\infty)$ by $X_n := \max(0,X_{n-1} + \vartheta_n)$, where $(\vartheta_n)_{n\geq1}$ is a sequence of $\R$-valued i.i.d.~random variables independent of $X_0$. 
The rate of convergence of Lindley's random walks with respect to $\|\cdot\|_V$ is investigated in \cite{Lun97,LunTwe96,Wu04}. More specifically, under the assumptions $\E[\gamma_0^{\vartheta_1}]<\infty$ for some $\gamma_0\in(1,+\infty)$ and $\E[\vartheta_1] < 0$, it is proved that there exists $\gamma\in(1,\gamma_0]$ such that $\E[\gamma^{\vartheta_1}] <1$, that $P$ is $V$-geometrically ergodic with $V(x) = \gamma^x$, and that $\rho_V(P) = \E[\gamma^{\vartheta_1}]$. The constant  $c_\rho$ of (\ref{strong-ergo}) associated with any $\rho\in(\E[\gamma^{\vartheta_1}],1)$ is not computed in  \cite{Lun97,LunTwe96,Wu04}. 

This subsection is devoted to the special case of discrete Lindley's random walks. More specifically, under the above  assumption, setting $V_\gamma:=(\gamma^n)_{n\in\N}$, we prove that the $V_\gamma$-geometric ergodicity property (\ref{strong-ergo}) holds true with the optimal rate $\rho:=\E[\gamma^{\vartheta_1}]$ and with an explicit (and simple) constant $c_\rho$. This result is based on Proposition~\ref{pi} and Proposition~\ref{ergo} using the distance\footnote{Such a distance is used in \cite[p.~296]{Wu04} to compute the rate of convergence of Lindley's random walks, but with no explicit constants.} $d(i,j) := |\gamma^i - \gamma^j|,\, $ $(i,j)\in\N^2$.  

Let $X_0$ be a $\N$-valued r.v.~and $(\vartheta_n)_{n\geq1}$ be a sequence of i.i.d.~$\Z$-valued r.v., independent of $X_0$. Let us introduce the sequence of $\N$-valued r.v.~$(X_n)_{n\in\N}$ defined by 
$$\forall n\geq1,\quad X_n := \max(0,X_{n-1} + \vartheta_n).$$ 
The common distribution $\nu:=(\nu_j)_{j\in\N}$ of the $\vartheta_n$'s is assumed to be such that 
\begin{equation} \label{mom-cont-ex-walk}
\exists\gamma_0\in(1,+\infty),\ \ \E[\gamma_0^{\vartheta_1}] = \sum_{j\geq 0} \nu_j\, \gamma_0^j < \infty \quad \text{ and } \quad \E[\vartheta_1] = \sum_{j\in\Z}j\, \nu_j < 0.
\end{equation}
\begin{pro} \label{pro-lind}
Under the assumptions~(\ref{mom-cont-ex-walk}), there exists $\gamma\in(1,\gamma_0]$ such that 
$$\kappa_1 := \E[\gamma^{\vartheta_1}] < 1,$$
and $(X_n)_{n\geq0}$ is $V_\gamma$-geometrically ergodic with $V_\gamma:=(\gamma^n)_{n\in\N}$. More precisely, we have the following properties: 
\begin{subequations}
\begin{eqnarray}
&\ & \forall f\in \cB_{V_\gamma},\ \forall n\geq1,\ \forall i\in\N, \quad |(P^nf)(i) - \pi(f)| \leq  c_1 \, {\kappa_1}^n\, m_1(f)\, \gamma^i \label{vit-lindley-fBV}\\
&\ & \forall (i,j)\in\N^2,\ \forall n\geq1, \quad \big|\P[X_n = j\, |\, X_0=i] - \pi(1_{\{j\}})\big| \leq  \frac{c_1\, \gamma^{i+1}}{(\gamma-1)\, \gamma^j}\ {\kappa_1}^n \label{vit-lindley} 
\end{eqnarray}
\end{subequations}
with $c_1 := \pi(V_{\gamma})$. Moreover, Inequality~(\ref{strong-ergo}) is fulfilled with $\rho:=\kappa_1$ and $c_\rho :=  \pi(V_{\gamma})(\gamma + 1)/(\gamma-1)$. 
\end{pro}
\begin{proof}{}
The first assertion holds since $G(\gamma) := \E[\gamma^{\vartheta_1}]$ satisfies $G(1)=1$ and $G\,'(1) = \E[\vartheta_1] <0$. To prove (\ref{vit-lindley-fBV})-(\ref{vit-lindley}), we apply Proposition~\ref{ergo} with the distance\footnote{The fact that, in Proposition~\ref{pro-lind}, the geometrical ergodicity is directly deduced from Proposition~\ref{ergo} is very particular. This is due to the choice of the distance in (\ref{choix-dist}).}
\begin{equation} \label{choix-dist}
\forall (i,j)\in\N^2,\quad d(i,j) := |\gamma^i - \gamma^j|.
\end{equation}

Note that we have with $x_0=0$: $\forall i\in\N,\ p(i) := 1+d(i,0) = \gamma^i$. Thus the space $\cB_{V_\gamma}$ corresponds to $\cB_1$ in (\ref{def-Bb}). Next observe that the spaces $\cB_1$ and $\cL_1$ coincide. Indeed, for all $f=(f(n))_{n\in\N}$ such that $|f|_1 := \sup_{n\in\N}|f(n)|/\gamma^{n} < \infty$, we have (use $\sup_{n\in\N^*} (\gamma^n + 1)/(\gamma^n-1) = (\gamma + 1)/(\gamma-1)$) 
\begin{equation} \label{comp-m1-weight}
m_1(f) := \sup\bigg\{\frac{|f(i)-f(j)|}{|\gamma^i - \gamma^j|},\, (i,j)\in\N^2,i\neq j\bigg\} \leq \frac{\gamma + 1}{\gamma-1}\, |f|_1. 
\end{equation}

Next we have: $\forall (v,i)\in\Z\times\N,\ F_vi := \max(0,i + v)$. Let us compute the Lipschitz (random) coefficient $L(\vartheta_1)$ with respect to the distance $d(i,j) := |\gamma^i - \gamma^j|$. We obtain for $(i,j)\in\N^2$ such that $i<j$ and for all $v\in\Z$: 
\begin{enumerate}[(a)]
\item $d\big(F_v i,F_v j\big) = \gamma^v\, \big|\gamma^i - \gamma^j\big|$ when $i + v\geq0$ and $j + v\geq0$;
\item $d\big(F_v i,F_v j\big) = \big|1 - \gamma^{j+v}\big| = \gamma^v\, \big|\gamma^{-v} - \gamma^j\big|$ when $i + v<0$ and $j + v\geq0$;
\item $d\big(F_v i,F_v j\big) = 0$ when $ i + v<0$ and $j + v<0$.
\end{enumerate}
In Case~(b), we have $i< -v \leq j$, thus $|\gamma^{-v} - \gamma^j| \leq |\gamma^i - \gamma^j|$. Thus 
$$L(v) := \sup_{(i,j)\in\N^2,i\neq j} \frac{d\big(F_v i,F_v j\big)}{|\gamma^i - \gamma^j|} =\gamma^v.$$
Finally, we obtain $\E[d(F_{\vartheta_1}0,0)] = \E[|\gamma^{\max(0,\vartheta_1)} - 1|] \leq \E[\gamma^{\vartheta_1}]$. Thus Conditions~(\ref{mom-cont-ex-walk}) implies that Conditions~$(\cC_{1})$ holds with $\E[L(\vartheta_1)] = \E[\gamma^{\vartheta_1}] <1$. Consequently, $P$ has an invariant distribution $\pi$ such that $\pi(V_{\gamma})<\infty$ from Proposition~\ref{pi}. Then, Property~(\ref{vit-lindley-fBV}) follows from (\ref{vit-lip-bis}) with $a:=1$ (note that $c_1=\int_\X p(y)\, d\pi(y)$). To obtain (\ref{vit-lindley}), use the fact that  $m_1(1_{\{j\}}) =  (\gamma-1)^{-1}\gamma^{1-j}$. 
\end{proof}
\begin{ex}[Simulation of a geometric distribution] \label{ex-metropolis}  
The Markov kernel $P$ on $\X:=\N$ defined for $p\in(0,1)$ by  
\begin{subequations}
\begin{eqnarray*}
& & \quad \quad \quad \quad \ \ P(0,0) := 1-p/2, \quad P(0,1) := p/2  \\ 
& & \forall i\in\N^*,\quad P(i,i-1) := 1/2,\quad P(i,i) := (1-p)/2,\quad P(i,i+1) := p/2. 
\end{eqnarray*}
\end{subequations}
arises from a Hastings-Metropolis sampler of a geometric distribution with parameter $p$. 
In \cite[Example~2]{MenTwe96}, $P$ is shown to be $V$-geometrically ergodic with $V=(p^{-n/2})_{n\geq0}$ and to have a convergence rate satisfying $\rho_V(P)\leq \rho_0 := \sqrt{p} + (1-p)/2$. More specifically: Property~(\ref{strong-ergo}) holds for any $\rho\in(\rho_0,1)$ with some constant $c_\rho$ such that $\lim_{\rho\r\rho_0} c_\rho = +\infty$. Proposition~\ref{pro-lind} allows us to improve this result. 

Indeed, $P$ can also be viewed as the Markov kernel of the discrete Lindley random walk  where $\vartheta_1$ is distributed as: $\P(\vartheta_1 = -1) =  1/2$, $\P(\vartheta_1 = 0) =  (1-p)/2$ and  $\P(\vartheta_1 = 1) =  p/2$. Consequently the estimates (\ref{vit-lindley-fBV})-(\ref{vit-lindley}) are valid for this  kernel, with here $\gamma:=p^{-1/2}$ and $\kappa_1 := \E[\gamma^{\vartheta_1}]=\rho_0$. In particular, Property~(\ref{strong-ergo}) holds true with $\rho:=\rho_0$ and $c_{\rho_0} := (1+\sqrt{p})^2/(1-\sqrt{p})$, namely we have for all $f=(f(n))_{n\in\N}$ such that 
$|f|_1 := \sup_n |f(n)|\, p^{n/2} < \infty$:
$$\forall n\geq1,\ \forall i\in\N,\quad p^{i/2}\, \big|P^nf(i) - \pi(f)\big| \leq \frac{(1+\sqrt{p})^2}{1-\sqrt{p}}\, |f|_1\, \left(\sqrt{p} + \frac{1-p}{2}\right)^n.$$
\end{ex}
%
\subsection{$V_a$-Geometrical ergodicity of IFS} \label{sub-rate-BV}  

Recall that we have set: $\forall x\in\X,\ V_a(x) := p(x)^a$. 
Let $(X_n)_{n\in\N}$ be  an IFS. Under Conditions~$(\cC_{a})$, Alsmeyer proved that, when $(X_n)_{n\in\N}$ is Harris recurrent and the support of $\pi$ has a non-empty interior, $(X_n)_{n\in\N}$ is $V_a$-geometrically ergodic, see  \cite[Prop.~5.2]{Als03}.  In Corollary~\ref{pro-D-K1} below, we prove that, under Conditions~$(\cC_{a})$, $(X_n)_{n\in\N}$ is $V_a$-geometrically ergodic with a convergence rate such that $\rho_{V_a}(P)\leq\widehat{\kappa}_a$, provided that $P^\ell : \cB_0\r\cB_a$ (for some $\ell\geq1$) is compact. The same result was proved in \cite[Prop.~7.2]{Wu04} under the following alternative hypotheses in place of our compactness assumption: $P$ and $P^N$ for some $N\geq 1$ are  Feller and strongly Feller respectively.

First observe that, under Conditions $(\cC_{a})$, Property~(\ref{vit-lip}) with $f:=V_a$ and $n:=1$ gives $PV_a\leq \xi_1\, V_a$ for some $\xi_1\in(0,+\infty)$, and so $P$ continuously acts on $\cB_a$. Second $P$ fulfills Condition~(\ref{cond-D}) with the function $V_a$. Indeed, let $\delta$ and $\kappa$ be such that $\widehat{\kappa}_a<\kappa<\delta<1$. Then there exists $N\in\N^*$ such that $c\, \kappa^N m_a(V_a) \leq \delta^N$, where $c\equiv c_\kappa$ is defined in (\ref{vit-lip}). Then Property~(\ref{vit-lip}) applied to $f:=V_a$ gives: $P^N V_a  \leq  \delta^N V_a + \pi(V_a)$. Since $\delta$ is arbitrarily close to $\widehat{\kappa}_a$, the real number $\delta_{V_a}(P)$ associated with $V_a$ via the definition (\ref{def-hat-rho}) satisfies: 
\begin{equation} \label{ifs-hat-rho}
\delta_{V_a}(P) 
\leq \widehat{\kappa}_a.
\end{equation}
\begin{cor} \label{pro-D-K1} 
Let us assume that Conditions~$(\cC_{a})$ hold true and that $P^\ell : \cB_0\r\cB_a$ (for some $\ell\geq1$) is compact. Then $P$ is $V_a$-geometrically ergodic, and we have 
\begin{equation} \label{rate-f-BV-bis}
r_{ess}(P) \leq \delta_{V_a}(P) \leq \widehat{\kappa}_a \quad \text{and} \quad \rho_{V_a}(P) \leq \widehat{\kappa}_a.
\end{equation}
\end{cor}
\begin{proof}{} The fact that $P$ is a power bounded quasi-compact operator on $\cB_a$ with $r_{ess}(P) \leq \delta_{V_a}(P) \leq \widehat{\kappa}_a$ follows from (\ref{ifs-hat-rho}) and Theorem~\ref{pro-qc-bis}. From this property, we deduce that the adjoint operator $P^*$ of $P$ is quasi-compact on the dual space $\cB_a'$ of $\cB_a$ and $r_{ess}(P^*) \leq r_{ess}(P)$.\footnote{Use Inequality~(\ref{ray-ess}), the fact that an operator and its adjoint have the same operator norms, and finally the fact that the adjoint of a compact operator is compact.} 
To avoid confusion, we denote by $P_{|{\cal L}_a}$ the restriction of $P$ on $\cL_a$. From Proposition~\ref{ergo} we know that $P_{|{\cal L}_a}$ is a bounded linear operator on $\cL_a$. 
Let us prove that $P$ is $V_a$-geometrical ergodic from Conditions~(b) in Theorem~\ref{CNS-qc-Vgeo}.

Let $r_0\in(\widehat{\kappa}_a,1)$. We show that $\lambda=1$ is the only eigenvalue of $P$ on $\cB_a$ such that $r_0  \le |\lambda| \leq 1$. Let $\lambda\in\C$ be such an eigenvalue. Then $\lambda$ is also an eigenvalue of $P^*$ since $P$ and $P^*$ have the same spectrum and $r_{ess}(P^*) \leq r_{ess}(P) \leq \widehat{\kappa}_a < |\lambda|$. Thus there exists $f\in\cB_a'$ such that $f'\circ P = \lambda f'$. But $f'$ is also in $\cL_a'$ since we have: $\forall f\in\cL_a,\ |\langle f',f\rangle| \leq \|f'\|_{{\cal B}_a'}|f|_a \leq \|f'\|_{{\cal B}_a'}\|f\|_a$. This proves that $\lambda$ is an eigenvalue of the adjoint of  $P_{|{\cal L}_a}$. Hence $\lambda$ is a spectral value of $P_{|{\cal L}_a}$. More precisely $\lambda$ is an eigenvalue of $P_{|{\cal L}_a}$ since, from (\ref{vit-lip-strong-ergo}), $P_{|{\cal L}_a}$ is quasi-compact on $\cL_a$ and $r_{ess}(P_{|{\cal L}_a}) \le \widehat{\kappa}_a < r_0 \le |\lambda|$. Finally we have $\lambda=1$. Indeed, if $\lambda\neq1$, then any $f\in\cL_a$ satisfying $Pf = \lambda f$ is such that $\pi(f)=0$, thus $f=0$ from (\ref{vit-lip-strong-ergo}) (pick $\kappa \in(\widehat{\kappa}_a, r_0))$. 

Now we prove that $1$ is a simple eigenvalue of $P$ on $\cB_a$. Using the previous property and the fact that $P$ is power bounded and quasi-compact on $\cB_a$, we know that $P^n\r\Pi$ in operator norm on $\cB_a$, where $\Pi$ is the finite rank eigen-projection on $\ker(P-I) = \ker(P-I)^2$. Set $m=\dim\ker(P-I)$. From \cite[Prop.~4.6]{Wu04} (see also \cite[Th.~1]{Her08}), there exist $m$ linearly independent nonnegative functions $f_1,\ldots,f_m\in\ker(P-I)$ and probability measures $\mu_1,\ldots,\mu_m\in\ker(P^*-I)$ satisfying $\mu_k(V_a)<\infty$ such that: $\forall f\in\cB_a,\ \Pi f = \sum_{k=1}^m \mu_k(f)\, f_k$. That $1$ is a simple eigenvalue of $P$ on $\cB_a$ then follows from Proposition~\ref{pi}. 

From Theorem~\ref{CNS-qc-Vgeo} and the previous results, for any $r_0\in(\widehat{\kappa}_a,1)$ we have $\rho_{V_a}(P) \leq r_0$. Thus we have $\rho_{V_a}(P)\leq \widehat{\kappa}_a$.   
\end{proof}

\begin{rem} 
In simple examples (as in Example~\ref{rem-IFS-3}), the optimal rate in (\ref{vit-lip}) is equal to $\widehat{\kappa}_a$. In this case, we have $\rho_{V_a}(P) =  \widehat{\kappa}_a$  since $\cL_a\subset\cB_a$. 
\end{rem}
\begin{rem}
The second inequality in (\ref{rate-f-BV-bis}) means that, for any real number $\kappa \in(\widehat{\kappa}_a,1)$, there exists a constant $e\equiv e_{\kappa}$ such that 
\begin{equation} \label{rate-f-BV-bis-bis}
\forall n\geq1,\ \forall f\in\cB_{a},\ \ |P^nf-\pi(f)\, 1_{\X}|_a \leq e\, \kappa^{n}\, |f|_a.
\end{equation}
Unfortunately the previous approach does not give any information on the constant $e$ of (\ref{rate-f-BV-bis-bis}). Inequality~(\ref{res-cor-f-Pf}) is more precise but in general is only valid for a smaller class of functions $f$. 
\end{rem}

\subsection{Applications to autoregressive models} \label{ex1-auto} 
Assume that $\X:=\R^q$ and denote the Lebesgue measure on $\R^q$ by $dy$. Let $\|\cdot\|$ denote any norm of $\R^q$, and define $d(x,y):=\|x-y\|$ the associated distance on $\R^q$. Set $p(x):=1+\|x\|$ ($x_0:=0$) and let us consider $V_a(x) := (1+\|x\|)^a$ with $a\in[1,+\infty)$. We have $\lim_{\|x\|\r +\infty}V_a(x)=+\infty$. We know from Remark~\ref{Rem_Noyau_Continu} that any Markov kernel $P(x,dy) = K(x,y)\, dy$, associated with a continuous (in the first variable) function $K : \R^q\times\R^q\r[0,+\infty)$, is compact from $\cB_0$ to $\cB_a$. This fact allows us to apply Corollary~\ref{pro-D-K1} to some classical IFSs. As an illustration, Properties~(\ref{res-cor-f-Pf}) and (\ref{rate-f-BV-bis}) are detailed below for affine autoregressive (AR) models. Such applications can be easily extended for others IFSs, as for instance for functional autoregressive models and AR processes with ARCH errors (see Examples~\ref{ex-fct-IFS}-\ref{ex-arch}). 
 
Let $(X_n)_{n\in\N}$ be the IFS 
\begin{equation} \label{auto-formule} 
X_0\in\R^q, \qquad \forall n\geq 1,\quad X_n := AX_{n-1} + \vartheta_n,
\end{equation}
associated with $F(v,x) := Ax +v$ where $A=(a_{ij})$ is a fixed real $q\times q$-matrix.   We have $L(v) = \|A\|$ where $\|A\|$ denotes the induced norm of $A$ corresponding to $\|\cdot\|$, and $d(F_v0, 0) = \|v\|$. Consequently, Conditions~$(\cC_{a})$ hold for $a\in[1,+\infty)$  provided that we have: 
\begin{equation} \label{auto-mom-cont} 
\|A\| < 1\ \quad \mbox{ and }\quad\E\big[\|\vartheta_1\|^a\big]<\infty.
\end{equation}
Under these conditions, $P$ has an invariant probability measure from Proposition~\ref{pi}, and we can easily prove that $\limsup_{\|x\|\r+\infty} PV_a(x)/V_a(x) \leq \|A\|^a$. Thus 
\begin{equation} \label{kappa_AR}
\widehat{\kappa}_a = \|A\| = \E[L(\vartheta_1)^{a}]^{\frac{1}{a}} \quad \text{ and }\quad \delta_{V_a}(P)\leq  \|A\|^a. 	
\end{equation}
The following result was already proved in \cite[Sect.~8]{Wu04}. 
\begin{cor} \label{cor-auto-qc} 
Assume that Conditions (\ref{auto-mom-cont}) hold true for some $a\in[1,+\infty)$  and that the common distribution of $(\vartheta_n)_{n\geq1}$ is absolutely continuous with respect to the Lebesgue measure on $\R^q$. Then $P$ is $V_a$-geometrically ergodic and if $r(A)$ denotes the spectral radius of $A$, then: 
$$r_{ess}(P) \leq r(A)^a \quad \text{and} \quad  \rho_{V_a}(P) = r(A).$$
\end{cor}
\begin{proof}{} 
Under the assumptions of the corollary, if $P^2$ is compact from $\cB_0$ to $\cB_a$, then we can deduced from Corollary~\ref{pro-D-K1} and (\ref{kappa_AR}) that $P$ is $V_a$-geometrically ergodic  and 
\[r_{ess}(P) \leq \|A\|^a \quad \text{and} \quad \rho_{V_a}(P) \leq \|A\|
.\]
Let us check that $P^2$ is compact from $\cB_0$ to $\cB_a$. Let $\nu(\cdot)$ denote the density of $\vartheta_1$. Note that $P$ has the form  (\ref{gene-kernel-K-P}) with the Lebesgue measure on $\R^q$ and $K(x,y) := \nu\big(y-Ax\big)$ for $(x,y)\in \R^q\times\R^q$. If $\nu(\cdot)$ is continuous, then the desired property follows from Remark~\ref{Rem_Noyau_Continu}. In the general case we can proceed as follows. Let $f\in\cB_0$ such that $\|f\|_0\leq 1$. Then we have
$$\forall (x,x')\in\R^q\times\R^q, \qquad \big|(Pf)(x') - (Pf)(x)\big| \leq \int_{\R^q} \big|\nu(y-A(x'-x)) - \nu(y)\big|\, dy.$$
Since $t\mapsto \nu(\cdot-t)$ is continuous from $\R^q$ to the Lebesgue space $\L^1(\R^q)$, it follows that $P$ is strongly Feller. Then Lemma~\ref{lem-fel-comp} gives the desired property. 

 Next, since $r_{ess}(P) \leq \|A\|^a$ and $\rho_{V_a}(P) \leq \|A\|$ whatever the norm, we obtain $r_{ess}(P) \leq r(A)^a$ and $\rho_{V_a}(P) \le r(A)$ using the fact that, for any $\varepsilon >0$, there is an induced norm $\| \cdot \|_{\varepsilon}$ such that $r(A)\le \| A \|_{\varepsilon} \le r(A) +\varepsilon$. The proof is complete provided that $\rho_{V_a}(P) \ge r(A)$. This follows from an easy computation based on the affine recursion (\ref{auto-formule}), see  \cite[top of p 301]{Wu04}. 
\end{proof}

That the model is $V_a$-geometrically ergodic under Conditions~(\ref{auto-mom-cont}) is well-known. However, to the best of our knowledge, the rates of convergence obtained in the next statement are new. Assertions $(i)$ and $(ii)$ below require stronger hypotheses than in Corollary~\ref{cor-auto-qc} and \cite[Prop.~8.3]{Wu04}, but they give rates of convergence with explicit constants. 
Recall that the total variation distance between two probability measures $\mu_1$ and $\mu_2$ on $\X$ is defined by $\|\mu_1-\mu_2\|_{TV} = \sup_{B\in{\cal X}} |\mu_1(B) - \mu_2(B)|$. The gradient is denoted by $\nabla$.

\begin{cor} \label{cor-auto-1} 
Assume that the assumptions of Corollary~\ref{cor-auto-qc} are fulfilled, that the density $\nu(\cdot)$ of  $\vartheta_1$ is continuously differentiable on $\R^q$, and that there exist some positive constants $\beta$ and $b$ such that 
\begin{equation} \label{auto-v-v'} 
\forall v\in\R^q,\qquad 
\|\nabla\nu(v)\| \leq \frac{b}{(1+\|v\|)^{\beta}}.
\end{equation}
Then the following assertions hold true: 
\begin{enumerate}[(i)]
  \item If $\beta > q + \gamma$ for some $\gamma\in[0,a-1]$, then for each $f\in\cB_{\gamma}$, we have $Pf\in\cL_a$ and 
\begin{equation} \label{AR-rate-f-alpha}
\forall n\geq2,\ \forall x\in\R^q,\quad \big|\E_x[f(X_n)] - \pi(f)\big| \leq c_1\, d_f\, \|A\|^{n-1}\, (1+\|x\|)^a,
\end{equation}
where $d_f := m_a(Pf)$ and $c_1$ is the constant of Proposition~\ref{ergo}. 
  \item If $\beta > q$ and if the initial distribution $\mu$ is such that $I_\mu := \int (1+\|x\|)^ad\mu(x) < \infty$, then: 
\begin{equation} \label{AR-rate-var-total} 
\forall n\geq2, \quad \|\P_\mu(X_n\in \cdot) - \pi(\cdot)\|_{TV} \leq c_1\, d_0\, I_\mu\, \|A\|^{n-1} 
\end{equation}
where $c_1$ is the constant of Proposition~\ref{ergo}, and the constant $d_0$ can be easily expressed in function of the matrix $A$ and the derivative of $\nu$ (in link with the norm $\|\cdot\|$). For instance, if $\|\cdot\|$ is the supremum norm on $\R^q$ then: $d_0 := q\big(\max_k\sum_{i=1}^q |a_{ik}|\big) \int_{\R^q} \|\nabla\nu(y)\|\, dy$.  
\end{enumerate}
\end{cor}
\begin{proof}{}
Recall that $P(x,dy)= K(x,y)\,dy$ with $K(x,y):=\nu\big(y-Ax\big)$ so that the partial derivative of $K$ in the direction $x$ satisfies: $\partial_x K(x,y) = - A^*\nabla\nu\big(y-Ax\big)$ where $A^*$ is the adjoint matrix of $A$.

Assertion~\emph{(i)} holds from Remark~\ref{cor-f-Pf}, if we prove that $P(\cB_\gamma)\subset\cL_a$. This is deduced from Proposition~\ref{CS} in Appendix~\ref{C} if we check Conditions~(\ref{K-dom-no-bounded})-(\ref{J-dom-eta}). From  (\ref{auto-v-v'}) it can be easily seen that, for any $r>0$, there exists a constant $a_r$ such that we have for all $x\in\R^q$ satisfying $\|x\| \leq r$: $\forall y\in\R^q,\ \ \| \partial_x K(x,y) \| \leq a_r\, (1+\|y\|)^{-\beta}$. Since $\beta - \gamma > q$, Condition~(\ref{K-dom-no-bounded}) holds. Next, set $\ J(x,y) := (1+\|y\|)^{\gamma}\, \| \partial_x K(x,y) \|$ for $y\in\R^q$. We have 
\begin{eqnarray}
\int_{\R^q} J(x,y) \, dy &=& \int_{\R^q} \big(1+\|Ax+v\|\big)^{\gamma}\, \|A^*\nabla\nu(v)\| \, dv \nonumber \\
&& \leq C\, \big(~(1 + \|x\|)^{\gamma}\, \int_{\R^q}\|\nabla\nu(v)\| \, dv + \int_{\R^q} \|v\|^{\gamma}\, \|\nabla\nu(v)\| \, dv\big) \nonumber  \\
&&\leq C'\, (1+\|x\|)^{\gamma},  \label{int-ineg-alpha}
\end{eqnarray}
which proves Condition~(\ref{J-dom-eta}). Thus Proposition~\ref{CS} gives $P(\cB_{\gamma})\subset\cL_a$. 

 Under the assumptions of \emph{(ii)}, setting $d_B := d_{1_B}$, we deduce from $(i)$ (case $\gamma=0$) that:
\begin{equation} \label{AR-rate-var-total-inter} 
\forall B\in\cX,\ \forall n\geq2,\ \forall x\in\R^q,\quad \big|\P_x(X_n\in B) - \pi(B)\big| \leq c_1\, d_B\, \|A\|^{n-1}\, (1+\|x\|)^a.
\end{equation}
Assuming (for simplicity) that $\|\cdot\|$ is the supremum norm on $\R^q$, it follows from (\ref{def-deri-Pf}) that we have for all $B\in\cX$: 
$$\forall x\in\R^q,\quad \big\|\nabla(P1_B)(x)\big\| \leq \big(\max_k\sum_{i=1}^q |a_{ik}|\big) \int_{\R^q} \|\nabla\nu(y-Ax)\|\, dy.$$
Thus $d_B := m_a(P1_B) \leq d_0$ with $d_0$ given in Corollary~\ref{cor-auto-1}, and (\ref{AR-rate-var-total}) easily follows from (\ref{AR-rate-var-total-inter}). 
\end{proof}
\begin{rem} \label{d(f)-alpha}
Under the conditions of Assertion~$(i)$, an upper bound of the constant $d_f := m_a(Pf)$ for $f\in\cB_{\gamma}$ can be easily derived (see Appendix~\ref{C} for details): 
\begin{gather*}
d_f := m_a(Pf) \, \leq\, q |f|_\gamma\, b\, C_{\gamma,\beta}\, \big(\max_k\sum_{i=1}^q |a_{ik}|\big)\\
\text{with }\  C_{\gamma,\beta} := \sup_{x\in\R^q} (1+\|x\|)^{-\gamma}\, \int_{\R^q}   \frac{(1+\|y+Ax\|)^\gamma}{(1+\|y\|)^\beta}\, dy < \infty. 
\end{gather*}
\end{rem}
\begin{rem} \label{rem-cste-AR1} 
The constant $c_1$ in (\ref{AR-rate-f-alpha})-(\ref{AR-rate-var-total}) is that of Proposition~\ref{ergo}. Let us give an upper bound of $c_1$ under Conditions~\emph{(\ref{auto-mom-cont})}. Set 
$$M := \E[\|\vartheta_1\|^a]^{1/a}, \quad \varepsilon_0 := \frac{1-\|A\|}{2}\quad \text{and} \quad r := \max\left(0,\frac{1+ M - \varepsilon_0}{\varepsilon_0}\right).$$ 
Recall that $x_0:=0$ here. Then we have for any $x\in\R^q$ such that $\|x\| \geq r$
\begin{eqnarray*}
\E\left[\left(\frac{1 + L(\vartheta_1)\, d(x,x_0) + d(F_{\vartheta_1}x_0,x_0)}{1+d(x,x_0)}\right)^a\right]^{\frac{1}{a}} &=& \E\left[\left(\frac{1 + \|A\|\, \|x\| + \|\vartheta_1\|}{1+\|x\|}\right)^a\right]^{\frac{1}{a}}  \\
&\leq& \frac{1 + \|A\|\, \|x\|}{1+\|x\|} + \frac{\E\big[\|\vartheta_1\|^a\big]^{1/a}}{1+r} \\
&\leq& \|A\| + \frac{1+M}{1+r} \leq  \frac{1+\|A\|}{2}.
\end{eqnarray*}
Set $\rho :=  \big((1+\|A\|)/2\big)^a$, $\xi_1 := \E\big[(\|A\| + \|\vartheta_1\|)^a\big]$ and  $\xi := 1+\xi_1(1+r)^a/(1-\rho)$. Recall that $\|\pi\|_b := \big(\int_\X (1+\|x\|)^b\, d\pi(y)\big)^{1/b}$ for $b:=1,a$. Then we have from Proposition~\ref{ergo} and Remark~\ref{rem-ergo-bis}
$$c_1 \leq \xi^{\frac{a-1}{a}}\, \|\pi\|_1 (1 + \|\pi\|_a)^{a-1}.$$ 
\end{rem}
\begin{ex}[Contracting normals] \label{ex2-rob-twe}
 Assume that $\X:=\R$ and that $P(x,dy)$ is the Gaussian distribution $\cN(\theta x,1-\theta^2)$ for $\theta\in(-1,1)$. This kernel is studied in \cite[Example~8.3]{Bax05}, in the specific case $\theta:=1/2$ in \cite{Kol00} and the convergence of its ergodic averages is discussed in \cite[Example~4]{RobTwe99}. Note that $P$ is the  transition kernel of the IFS defined by $\forall n\geq 1,\ X_n := \theta X_{n-1} + \vartheta_n$, where $(\vartheta_n)_{n\geq 1}$ is a sequence of $\R$-valued i.i.d.~random variables, with common distribution $\cN(0,1-\theta^2)$. It can be easily checked that $P$ has the standard Gaussian distribution $\cN(0,1)$ as invariant probability measure $\pi$. Here we have $\cB_a:=\{f : \R\r\C,\ \sup_{x\in\R}|f(x)|/(1+|x|)^a < \infty\}$. 
 
Let $a\in[1,+\infty)$. Since Conditions~(\ref{auto-mom-cont}) hold, we have $r_{ess}(P) \leq |\theta|^a$ with $P$ considered as an operator on $\cB_a$. Concerning the geometric ergodicity, Corollary~\ref{cor-auto-qc} ensures that the convergence rate of $P$ on $\cB_a$ satisfies $\rho_{V_a}(P) = |\theta|$ for any $a\in[1,+\infty)$ (also use Example~\ref{rem-IFS-3} to obtain the last  equality). This improves all the earlier bounds obtained for $\rho_{V_a}(P)$ in this example (compare with \cite{Bax05} in case $a:=2$). 

Furthermore, for this example, Property~(\ref{AR-rate-var-total}) enables us to improve and simplify the results of \cite[Section~5]{RobTwe99} concerning the total variation convergence bounds. In fact, for any initial distribution $\mu$ such that $I_\mu := \int (1+|x|) d\mu(x) < \infty$, Inequality~(\ref{AR-rate-var-total}) with $a:=1$ gives (use $c_1 := 1+\sqrt{2/\pi}$, $d_0 := 2|\theta|/\sqrt{2\pi(1-\theta^2)}$): 
$$
\forall n\geq2, \quad \|\mu P^n - \pi\|_{TV} \leq  \frac{\sqrt{2\pi}+2}{\pi \sqrt{1-\theta^2}} \, I_\mu\, |\theta|^{n}. 
$$

Finally Assertion~(i) of Corollary~\ref{cor-auto-1} provides an interesting alternative result between the last one and the $V_a$-geometrical ergodicity. For instance, in case $a:=2$, Property~(\ref{AR-rate-f-alpha}) ensures that, for all $f\in\cB_1$, we have $Pf\in\cL_2$ and: 
\begin{equation} \label{rate-b1-b2}
\forall n\geq2,\ \forall x\in\R^q,\quad \big|\E_x[f(X_n)] - \pi(f)\big| \leq c\, d_f\, |\theta|^{n}\, (1+|x|)^2,
\end{equation}
with $c := 2\, \big(1+ \sqrt{2}/\sqrt{\pi}\big)\big(1+\sqrt{2+ 2\sqrt2/\sqrt\pi}\ \big)/|\theta|$ (first use that $P^nW_2(x)=2(1-\theta^{2n})+\theta^{2n}W_2(x)$ for $n\ge 1$ with $W_2(x) = 1+x^2$, so that $\sup_{n\ge 1}\sup_{x\in\R} P^nW_2(x)/W_2(x) =2$; second deduce from $V_2/2\le W_2\le V_2$ that $\xi \le 4 $; third check that $\|\pi\|_1=1+\sqrt{2/\pi}$, ${\|\pi\|_2}^2:=2(1+\sqrt{2/\pi})$). Note that (\ref{rate-b1-b2}) does not involve the $V_a$-geometrical ergodicity (either on $\cB_1$, or on $\cB_2$). However the rate of convergence in (\ref{rate-b1-b2}) is optimal and the associated constant $C$ is explicit.  The weighted-Lipschitz constant $d_f := m_2(Pf)$ can be easily computed thanks to Remark~\ref{d(f)-alpha}.
\end{ex}
The two last examples are classical extensions of the affine ARs. 

\begin{ex}[The functional autoregressive process] \label{ex-fct-IFS}
Let $(X_n)_{n\in\N}$ be the IFS  
\begin{equation} \label{auto-formule-fonctionel} 
X_0\in \R^q, \qquad \forall n\geq 1,\quad X_n := \psi(X_{n-1}) + \vartheta_n,
\end{equation}
associated with $F(v,x) := \psi(x) + v$ where $\psi : \R^q\r\R^q$ is a fixed differentiable function and $\vartheta_1$ has a density $\nu(\cdot)$. Assume that 
\begin{equation} \label{auto-mom-cont-fonctionel}
\alpha := \sup_{x\in\R^q}\|\psi'(x)\| < 1 \ \ \ \mbox{and}\ \ \ \exists a\in(1,+\infty),\  \E\big[\|\vartheta_1\|^a\big]<\infty. 
\end{equation}
Then Corollary~\ref{cor-auto-qc} extends to the IFS model (\ref{auto-formule-fonctionel}) as follows: $r_{ess}(P) \leq \alpha^a$ and $\rho_{V_a}(P) \leq \alpha$. A direct adaptation of the above arguments allows us to prove that, if the density $\nu(\cdot)$ of  $\vartheta_1$ satisfies \emph{(\ref{auto-v-v'})}, then all the conclusions (i) and (ii) of Corollary~\ref{cor-auto-1} hold true with again $\alpha$ in place of $\|A\|$. 
\end{ex}

\begin{ex}[Autoregressive process with ARCH(1) error] \label{ex-arch} 
Let $(X_n)_{n\in\N}$ be the real-valued IFS 
\begin{equation} \label{auto-arch} 
X_0\in\R, \qquad \forall n\geq 1,\ \ \ X_n := aX_{n-1} + \sigma(X_{n-1})\vartheta_n,
\end{equation}
where $\sigma : \R\r\R$ is defined by $\sigma(x):= \sqrt{b+cx^2}$ with fixed $a,b>0$. It can be easily seen that the associated Markov kernel is of the form $P(x,dy) = K(x,y)dy$, provided that the probability distribution of $\vartheta_1$ has a density $\nu(\cdot)$. Corollaries~\ref{cor-auto-qc} and \ref{cor-auto-1} can be adapted under suitable assumptions on $\nu(\cdot)$. 
\end{ex}

\newpage

\appendix

\section{Positive eigenvectors of the adjoint of a nonnegative operator on $\cB_V$} \label{B}
\begin{apro} If $L$ is a positive bounded linear operator on $\cB_V$ such that $r(L)=1$, then there exists a nontrivial nonnegative continuous linear form $\eta$ on $\cB$ such that $\eta\circ L = \eta$. 
\end{apro}
\begin{proof}{}  Since $r(L)=1$ and the spectrum $\sigma(L)$ of $L$ is closed, there exists $\lambda\in\sigma(L)$ such that $|\lambda| = 1$. Set $\lambda_n := \lambda(1+1/n)$. From the Banach-Steinhaus theorem, there exists $f_0\in\cB_V$, $f_0\geq 0$,  
such that $\|(\lambda_nI-L)^{-1}f_0\|_V\r+\infty$ when $n\r+\infty$. Using the Neumann series, namely 
$$|z| >1\ \Rightarrow\ (zI-L)^{-1} = \sum_{k\geq 0} z^{-(k+1)}\, L^k,$$ 
the positivity of $L$ (which gives $|L^kf_0| \leq L^k|f_0|$), and finally the fact that $\cB_V$ is a Banach lattice 
($\forall (f,g)\in\cB_V^2 : |f| \leq |g| \Rightarrow \|f\|_V = \|\, |f|\, \|_V  \leq \|\, |g|\, \|_V = \|g\|_V$), we obtain: 
$$\|(\lambda_nI-L)^{-1}f_0\|_V \leq \|(|\lambda_n|I-L)^{-1}|f_0|\, \|_V\, \r+\infty\ \quad \mbox{as}\ n\r+\infty.$$
This implies that $1\in\sigma(L)$. Now, let $\cB_V'$ denote the dual space of $\cB_V$ with the associated norm also denoted by $\|\cdot\|_V$, and let $L^*$ be the adjoint of $L$. 
Since $\sigma(L^*) = \sigma(L)$, we have $1\in\sigma(L^*)$. 

Let us set $\beta_n := 1+1/n$. We deduce from the Banach-Steinhaus theorem that there exists $f_0'\in\cB_V'$, $f_0'\geq 0$, (i.e.~$\forall f\in\cB_V : f\geq0\ \Rightarrow\ f_0'(f) \geq 0$) such that 
$$b_n := \|(\beta_n\, I-L^*)^{-1}\, f_0'\|_V\r+\infty \quad  \mbox{as}\ n\r+\infty,$$ 
where $I$ denotes here the identity map on $\cB_V'$. Let us define the following positive elements in the unit ball of $\cB_V'$: 
$$f'_n := 
\frac{1}{b_n}\, (\beta_n\, I-L^*)^{-1}f_0' = \frac{1}{b_n}\, \sum_{k\geq 0} \beta_n^{-(k+1)}\, (L^*)^kf_0' \ \ \ \ (n\in\N^*).$$
We have $f'_n\geq0$, thus $\|f'_n\|_V = f'_n(V) = 1$. 
Thanks to the Banach-Alaoglu theorem, the sequence $(f'_n)_n$ has a limit point, say $\eta$, in the unit ball of $\cB_V'$ for the weak topology in $\cB_V'$, that is: for all neighborhood $W$ of $\eta$ and for all $N\geq1$, there exists $n > N$ such that $f'_n\in W$. Now, given $f\in\cB_V$, consider the following special neighborhoods of $\eta$ with respect to the weak topology of $\cB_V'$:
$$W(f,p) = \left\{f'\in\cB_V' : \big|f'(f) - \eta(f)\big| < \frac{1}{p},\ \big|f'(Lf) - \eta(Lf)\big| < \frac{1}{p}\right\}.$$
Let us denote by $(n_p)_p \equiv(n_p(f))_p$ any increasing sequence of integer numbers such that we have $f'_{n_p}\in W(f,p)$ for all $p\geq 1$. First, pick any $f\in\cB_V$, $f\geq 0$: then  it follows from $f'_{n_p}(f)\geq 0$ and $|f'_{n_p}(f) - \eta(f)| < 1/p$ that $\eta(f)\geq 0$. So $\eta\geq 0$. Second, consider $f:=V$:  then we obtain $\eta(V) = 1$ from $f'_{n_p}(V) = 1$ and 
$|f'_{n_p}(V) - \eta(V)| < 1/p$, so $\eta\neq0$. Finally, let us fix any $f\in\cB_V$. We have 
\begin{eqnarray*}
\beta_n(f'_n-\eta)(f) - (f'_n-\eta)(Lf) &=& (\beta_n\, I-L^*)(f'_n)(f) - (\beta_n\, I-L^*)(\eta)(f) \\
&=& 
\frac{1}{b_n} f_0'(f) - (\beta_n\, I-L^*)(\eta)(f). 
\end{eqnarray*}
Replacing $n$ with $n_p\equiv n_p(f)$ gives  $(I-L^*)(\eta)(f) = 0$ as $p\r+\infty$. Namely: $\eta\circ L = \eta$. 
\end{proof}

\section{Proof of Lemma~\ref{lem-fel-comp} and Lemma~\ref{lem-exist-proba-inv}}  \label{ap-exist-pi}
\noindent \begin{proof}{ of Lemma~\ref{lem-fel-comp}} Obviously, we may assume $\ell:=1$. 
Let $(f_n)_{n\in\N}\in\cB_0^{\N}$ such that $\|f_n\|_0\leq1$. From the Banach-Alaoglu theorem, there exist a subsequence  $(f_{n_k})_{k\in\N}$ and $f\in\L^\infty(\eta)$ such that 
$$\forall g\in\L^1(\eta),\quad \lim_k \int f_{n_k}\, g\, d\eta = \int f\, g\, d\eta.$$
Since $K(x,\cdot)\in\L^1(\eta)$ for all $x\in\X$, we obtain: $\forall x\in\X,\ \lim_k (Pf_{n_k})(x) = (Pf)(x)$. Define  
$$\forall k\geq 1,\ \forall x\in\X,\quad \Delta_k(x) := \sup_{p,q\geq k} \big|(Pf_{n_q})(x)-(Pf_{n_p})(x)\big|.$$
The sequence $(P\Delta_k)_{k\ge 1}$ is non increasing from $\Delta_{k+1}\leq \Delta_k$ and the positivity of $P$. Next, since $\Delta_k\r 0$ (pointwise) and $\|\Delta_k\|_\infty\leq2$, we have $P\Delta_k\searrow 0$ (pointwise) from Lebesgue's theorem. Note that each $P\Delta_k$ is continuous from the strong Feller assumption. Then we deduce from the monotone Dini theorem that the sequence $(P\Delta_k)_{k\ge 1}$ uniformly converges to $0$ on each compact of $\X$. Now let $\varepsilon>0$. Then there exists $A>0$ such that $d(x,x_0) > A\, \Rightarrow\, V(x)^{-1} < \varepsilon/2$, and there exists $n_0\in\N$ such that: $\forall k\geq n_0,\ \sup_{d(x,x_0)\leq A} |(P\Delta_k)(x)| <  \varepsilon$. Since $V\geq1$ and $\|P\Delta_k\|_0\leq2$, we obtain for every  $k\geq n_0$: 
$$\|P\Delta_k\|_V = \sup_{x\in\X}\frac{|(P\Delta_k)(x)|}{V(x)} \leq \max\bigg(\sup_{d(x,x_0)\leq A} |(P\Delta_k)(x)|, \sup_{d(x,x_0) > A} \frac{|(P\Delta_k)(x)|}{V(x)}\bigg) < \varepsilon.$$
Finally observe that we have for all $p,q\geq k$ 
$$\big|P^2f_{n_q} - P^2f_{n_p}\big| \leq P\big(|Pf_{n_q} - Pf_{n_p}|) \leq P\Delta_k,$$
therefore we have: $\forall p,q\geq n_0,\ \|P^2f_{n_q} - P^2f_{n_p}\|_V \leq \|P\Delta_{n_0}\|_V< \varepsilon$. We have proved that the sequence $(P^2f_{n_k})_{k\ge 1}$ is Cauchy in $\cB_V$. Hence it converges in $\cB_V$. 
\end{proof}

\noindent \begin{proof}{ of Lemma~\ref{lem-exist-proba-inv}}
We know from (\ref{cond-D}) that $P$ is power-bounded on $\cB_V$. Let $x_0\in\X$. Then we have $K:=\sup_n(P^nV)(x_0) < \infty$. Let $\pi_n$, $n\geq 1$, be the probability measure on $(\X,\cX)$ defined by: $\forall B\in\cX,\ \pi_n(1_B) = \frac{1}{n}\sum_{k=0}^{n-1}(P^k1_B)(x_0)$. Then Markov's inequality gives 
$$\forall n\geq 1,\ \forall \alpha\in(0,+\infty),\quad \pi_n\big(1_{\{V>\alpha\}}\big) \leq \frac{\pi_n(V)}{\alpha} \leq \frac{K}{\alpha}.$$ 
Thus the sequence $(\pi_n)_{n\ge 1}$ is tight, and we can select a subsequence $(\pi_{n_k})_{k\in\N}$ weakly converging to a probability measure $\pi$, which is clearly $P$-invariant. Next, for $p\in\N^*$, define $V_p(\cdot) = \min(V(\cdot),p)$. Then 
$\forall k\geq 0,\ \forall p\geq 0,\ \pi_{n_k}(V_p) \leq \pi_{n_k}(V) \leq K$. Since $V_p$ is continuous and bounded on $\X$, we obtain: $\forall p\geq 0,\ \lim_k\pi_{n_k}(V_p) = \pi(V_p) \leq K$. The monotone convergence theorem then gives $\pi(V) < \infty$. 
\end{proof}

\section{Additional material for discrete Markov chains}  \label{annexe_discret}

\subsection{Complements for the proof of Corollary~\ref{cor-geo-I-A}} \label{app-compl-cor3}
\begin{lem} \label{IA_spectre}
Under Conditions~\emph{(\ref{cond-ireduc})}-\emph{(\ref{cond-aperiod})}, $1$ is a simple eigenvalue and the unique eigenvalue of modulus one of $P$ on $\cB_V$. 
\end{lem} 
\begin{proof}{}
First prove that the support of $\pi$ coincides with $\N$. We have 
$$\forall j\in\N,\ \forall n\geq 1,\quad \pi(j) := \pi(P^n1_{\{j\}}) = \sum_{i\geq0} \pi(i)\, P^{n}(i,j).$$
Hence, if $\pi(j)=0$ for some $j\in\N$, then we would obtain $\pi(i)=0$ whenever $P^{n}(i,j)>0$, thus $\pi(i)=0$ for all $i\in\N$ from Condition~(\ref{cond-ireduc}), which is impossible. 

Second, we have the following implication: $\forall \lambda\in\C,\ |\lambda|=1,\ \forall f\in\cB_V$, 
\begin{equation} \label{implic-f-mod-f}
Pf=\lambda f\ \Rightarrow\ P|f|=|f|.
\end{equation}
Indeed, we deduce that $|f|\leq P|f|$ from $Pf=\lambda f$ and the positivity of $P$. Then it follows from $\pi(P|f|-|f|)=0$ that $P|f|=|f|$ $\pi$-a.s. Thus  $P|f|=|f|$ since the support of $\pi$ is $\N$.

Third, we prove that 1 is a simple eigenvalue of $P$ on $\cB_V$. Let $g\in\cB_V$ such that $Pg=g$, and set $f:=g-g(0)1_\N$. Then $Pf=f$ so that $P|f|=|f|$. We have: $\forall n\geq 1,\ 0 = |f(0)| = \sum_{j\geq0} P^{n}(0,j)\, |f(j)|$. Condition~(\ref{cond-ireduc}) then yields $f\equiv0$, namely $g$ is constant. 

Finally, let $\lambda\in\C$, $|\lambda|=1$, and let $f\in\cB_V$, $f\neq0$, be such that $Pf=\lambda f$. It follows from (\ref{implic-f-mod-f}) and the last statement that, $\forall n\in\N,\ |f(n)|=1$ (up to a multiplicative constant). From $|\lambda|=1$, $|f|\equiv1$, and $\forall n\geq 1,\ \lambda^n\, f(0) = \sum_{j\geq0} P^{n}(0,j)\, f(j)$, we obtain: $P^{n}(0,j)>0 \Rightarrow \lambda^n\, f(0) = f(j)$. In particular: $n\in\cR_{0,0} \Rightarrow \lambda^n\, f(0) = f(0)$. This gives: $\forall(m,n)\in\cR_{0,0}\times\cR_{0,0},\ \lambda^{n-m} = 1$, hence $\lambda=1$ by Condition~(\ref{cond-aperiod}). 
\end{proof}

\subsection{Random walks with bounded increments} \label{proof_prop1_NHRW}

\subsubsection{Proof that the integer $\ell$ in (\ref{def-ell-RRW}) is well-defined} \label{EllWellDef}
   Set $A(\gamma) := \phi(\gamma)\gamma^b=\sum_{k=0}^{2b} a_{-b+k}\gamma^k$ where $\phi(\gamma)$ is defined in (\ref{non-hom-cont}) and $a_0\neq 1$. The integer $\ell$ in (\ref{def-ell-RRW}), if well-defined, can be equivalently characterized from Leibniz's formula by
$$\forall k\in\{1,\ldots,\ell-1\},\quad \phi^{(k)}(1) = A^{(k)}(1) - \prod_{j=0}^{k}(b-j) =0 \quad \text{and} \quad \phi^{(\ell)}(1) =  A^{(\ell)}(1) - \prod_{j=0}^{\ell-1}(b-j)\neq 0,$$
according that the first condition is removed when $\ell=1$.  To prove the existence of such an integer $\ell$, observe that, if $A^{(k)}(1) = \prod_{j=0}^{k-1}(b-j)$ for $k=1,\ldots,2b$, then Taylor's formula would give  
$$A(\gamma) = A(1) + \sum_{k=1}^{2b} \bigg(\prod_{j=0}^{k-1}(b-j)\bigg)\frac{(\gamma-1)^k}{k!} = 1 + \sum_{k=1}^{b} \binom{b}{k}(\gamma-1)^k = \gamma^b,$$
which is impossible since $a_0\neq1$.

\subsubsection{Condition~(\ref{cond-D}) for random walks with i.d.~bounded increments} \label{app-equi-WD}
Let $P$ be defined on $\X=\N$ by 
\begin{gather*}
\forall i\in\{0,\ldots,b-1\},\quad \sum_{j\ge 0} P(i,j)=1; \quad 
\forall i\ge b, \forall j\in\N, \quad P(i,j) = \left\{ \begin{array}{lcl}
a_{j-i}  & \text{if} & |i-j|\le b  \\
0 & \text{if}& |i-j|> b
\end{array}\right.
\end{gather*} 
where $b\in\N^*$, $(a_{-b},\ldots,a_b)\in[0,1]^{2b+1}$ and $\sum_{k=-b}^{b} a_k=1$. Assume that $a_0<1$ and that there exists $\gamma\in(1,+\infty)$ such that 
$$\forall i\in\{0,\ldots,b-1\},\quad \sum_{j\ge 0} P(i,j)\gamma^j < \infty.$$
The integer $\ell$ in the next proposition is defined by (\ref{def-ell-RRW}), it is well-defined from Subsection~\ref{EllWellDef}.
\begin{pro} \label{pro-hom-discrete}
The following conditions are equivalent: 
\begin{enumerate}[(a)]
	\item There exists $\gamma_0\in(1,\gamma]$ such that $P$ satisfies Condition \emph{(\ref{cond-D})} with $V_{\gamma_0} := ({\gamma_0}^n)_{n\in\N}$, and we have 
	$$\delta_{V_{\gamma_0}}(P) = \frac{A(\gamma_0)}{\gamma_0^g}= \phi(\gamma_0);$$ 
	\item $\displaystyle A^{(\ell)}(1) < \prod_{j=0}^{\ell-1}(b-j)$, i.e. $\phi^{(\ell)}(1)<0$.
\end{enumerate}
\end{pro}
\begin{proof}{} 
Let us prove the equivalence $(a)\Leftrightarrow (b)$. Assume that $A^{(\ell)}(1) > \prod_{j=0}^{\ell-1}(b-j)$, i.e. $\phi^{(\ell)}(1)>0$, and prove that, for all $\gamma_0\in(1,\gamma]$, $P$ does not satisfy (\ref{cond-D}) with $V:=V_{\gamma_0}$. From the definition of $\ell$ and from $\phi^{(\ell)}(1)>0$, there exists $\gamma_2\in(1,\gamma]$ such that $\phi(\gamma') > \phi(1)=1$ for all $\gamma'\in(1,\gamma_2)$, so that (\ref{itere-V-gamma}) gives  
$$\forall \gamma'\in(1,\gamma_2),\ \forall N\geq1,\quad \limsup_{n\r+\infty} \frac{(P^NV_{\gamma'})(n)}{V_{\gamma'}(n)} = \phi(\gamma')^N >1.$$
Hence, from Corollary~\ref{cor-qc-bis}, for all $\gamma'\in(1,\gamma_2)$, $P$ does not satisfy Condition~(\ref{cond-D}) with $V:=V_{\gamma'}$. This proves the desired result. Indeed, if (\ref{cond-D}) holds with $V:=V_{\gamma_0}$ for some $\gamma_0\in(1,\gamma]$, then (\ref{cond-D}) would be fulfilled for all $\gamma'\in(1,\gamma_0]$ from Jensen's inequality, which contradicts the last conclusion. 

Conversely, assume that $\phi^{(\ell)}(1)<0$. Then there exists $\gamma_1\in(1,\gamma]$ such that $\phi(\gamma_0) < 1$ for all $\gamma_0\in(1,\gamma_1)$. Let $\gamma_0\in(1,\gamma_1)$. From (\ref{itere-V-gamma}) with $k=1$, we obtain  $(PV_{\gamma_0})(n) = \phi(\gamma_0)\, V_{\gamma_0}(n)$ for all $n\geq b$. Since by assumption we have $\sum_{j\ge 0} P(i,j)\gamma_0^j <\infty$ for each $0\leq i\leq b-1$, it follows from Corollary~\ref{cor-qc-bis} that 
$P$ satisfies (\ref{cond-D}) with $V:=V_{\gamma_0}$ and that $\delta_{V_{\gamma_0}}(P)\leq\phi(\gamma_0)$. The converse inequality follows from (\ref{itere-V-gamma}) and Corollary~\ref{cor-qc-bis}. 
\end{proof}

\section{Proof of Formula~(\ref{vit-lip-bis})} \label{BB}

For the sake of simplicity we prove Properties~(\ref{vit-lip}) and (\ref{vit-lip-strong-ergo}) with explicit constants in the special case when  $\kappa_1 := \E[L(\vartheta_1)^{a}]^{1/a} < 1$. Under the general assumption $\widehat{\kappa}_a<1$ of Conditions $(\cC_{a})$, the proof of (\ref{vit-lip})-(\ref{vit-lip-strong-ergo}) is similar (replace $P$ by $P^N$ with $N$ such that $\E[L(F_{\vartheta_N:\vartheta_1})^{{a}}] < 1$). 

First, we prove that the constant $\xi$ in Proposition~\ref{ergo} is well defined. Second, we obtain a basic estimate (\ref{cor-ineg-rate-gene}) of the distance between functional of the states  occupied at time $n$ of the IFS from two different initial probability distributions. Then, we complete the proof of Formula~(\ref{vit-lip-bis}). 

First, we have for any $ x\in\X$
\begin{eqnarray*}
\frac{(PV_a)(x)}{V_a(x)} = \E\left[\left(\frac{1+d(F_{\vartheta_1}x,x_0)}{1+d(x,x_0)}\right)^a\right] 
&\leq& \E\left[\left(\frac{1 + d(F_{\vartheta_1}x,F_{\vartheta_1}x_0) + d(F_{\vartheta_1}x_0,x_0)}{1+d(x,x_0)}\right)^a\right] \\
&\leq& \E\left[\left(\frac{1 + L(\vartheta_1)\, d(x,x_0)}{1+d(x,x_0)} + \frac{d(F_{\vartheta_1}x_0,x_0)}{1+d(x,x_0)}\right)^a\right].
\end{eqnarray*}
Since $\E[(\max(1,L(\vartheta_1)) + d(F_{\vartheta_1}x_0,x_0))^a] < \infty$, we obtain $\xi_1 := \sup_{x\in\X}(PV_a)(x)/V_a(x) < \infty$.  
Next Lebesgue's theorem ensures that $\limsup (PV_a)(x)/V_a(x)$ converges to $\E[L(\vartheta_1)^a]$ when $d(x,x_0)\r+\infty$. Now let $\delta$ be such that $\E[L(\vartheta_1)^{{a}}] < \delta < 1$. Then there exists $r>0$ such that we have for all $x\in\X$ satisfying 
$d(x,x_0) >  r$: $PV_{a}(x)\leq  \delta V_{a}(x)$. Besides, if $d(x,x_0) \leq  r$, then we obtain $(PV_{a})(x) \leq \xi_1 V_a(x) \leq \xi_1(1+r)^a$.  
Thus: $PV_{a} \leq \delta V_{a} + \xi_1(1+r)^a1_\X$. Therefore 
\begin{equation} \label{d2} 
P^nV_{a} \leq \delta^n\, V_{a} + \frac{\xi_1(1+r)^a}{1-\delta}1_\X \leq \big(1+\frac{\xi_1(1+r)^a}{1-\delta}\big)\, V_{a}. 
\end{equation}
This prove that the bound $\xi$ given in Proposition~\ref{ergo} is finite. 

Second, let us introduce some additional notations. If $\mu$ is a probability measure  on $\X$ and $X_0\sim\mu$, we make a slight abuse of notation in writing $(X_n^\mu)_{n\in\N}$ for the associated IFS. We simply write $(X_n^x)_{n\in\N}$ when $\mu:=\delta_x$ is the Dirac mass at some $x\in\X$. We denote by $\cM_a$ the set of all the probability measures $\mu$ on $\X$ such that $\|\mu\|_a := (\int_\X V_a(y)\, d\mu(y))^{1/a} < \infty$. 
Finally, for $n\in\N$ and for any probability measures $\mu_1$ and $\mu_2$ on $\X$, define:  
$$\Delta_n(\mu_1,\mu_2) := d\big(X_n^{\mu_1},X_n^{\mu_2}\big)\, \big(p(X_n^{\mu_1}) + p(X_n^{\mu_2})\big)^{a-1}.$$
\begin{alem} \label{pro-rate-gene}
We have: $\forall n\geq1,\ \forall (\mu_1,\mu_2)\in \cM_{a}\times\cM_{a}$
\begin{equation} \label{ineg-rate-gene}
\E\big[\Delta_n(\mu_1,\mu_2)\big] \leq \xi^{\frac{a-1}{a}}\, \kappa_1^n\, \E[d(X_0^{\mu_1},X_0^{\mu_2})]\, \big(\|\mu_1\|_a + \|\mu_2\|_a\big)^{a-1}.
\end{equation}
Furthermore we have for all $f\in\cL_{a}$: 
\begin{equation} \label{cor-ineg-rate-gene}
\E\big[|f(X_{n}^{\mu_1}) - f(X_n^{\mu_2})|\big] \leq  \xi^{\frac{a-1}{a}}\,  m_{a}(f)\, \kappa_1^n\, \E[d(X_0^{\mu_1},X_0^{\mu_2})]\, \big(\|\mu_1\|_a + \|\mu_2\|_a\big)^{a-1}. 
\end{equation}
\end{alem}
\begin{proof}{}
If $a=1$, then (\ref{ineg-rate-gene}) follows from the independence of the $\vartheta_n$'s and from the definition of $L(v)$ and $\kappa_1$. Now assume that $a\in(1,+\infty)$. Without loss of generality, one can suppose that the sequence $(\vartheta_n)_{n\geq1}$ is independent from $(X_0^{\mu_1},X_0^{\mu_2})$. Also note that, if $\mu\in \cM_{a}$, then we have 
$$\E\big[p(X_n^{\mu})^a\big] = \int_\X (P^nV_a)(x) d\mu(x) \leq  \xi\, \|\mu\|_a^a.$$  
From Holder's inequality (use $1 = 1/a + (a-1)/a$), we obtain  
\begin{eqnarray*} 
\E\big[\Delta_n(\mu_1,\mu_2)\big] &=& \E\left[d\big(F_{\vartheta_n:\vartheta_1}X_0^{\mu_1},F_{\vartheta_n:\vartheta_1}X_0^{\mu_2}\big)\, \big(p(X_n^{\mu_1}) + p(X_n^{\mu_2})\big)^{a-1}\right] \\
&\leq& \E[d(X_0^{\mu_1},X_0^{\mu_2})]\, \E\left[L(\vartheta_n:\vartheta_1)\, \big(p(X_n^{\mu_1}) + p(X_n^{\mu_2})\big)^{a-1} \right] \\
&\leq& \E[d(X_0^{\mu_1},X_0^{\mu_2})]\, \E\big[L(\vartheta_n:\vartheta_1)^{a}\big]^{\frac{1}{a}}\, \E\big[\big(p(X_n^{\mu_1}) + p(X_n^{\mu_2})\big)^{a}\big]^{\frac{a-1}{a}} \\
&\leq& \E[d(X_0^{\mu_1},X_0^{\mu_2})]\, \E\big[L(\vartheta_1)^{a}\big]^{\frac{n}{a}}\, \xi^{\frac{a-1}{a}}\, (\|\mu_1\|_a + \|\mu_2\|_a)^{a-1}.
\end{eqnarray*}
This proves (\ref{ineg-rate-gene}). Property~(\ref{cor-ineg-rate-gene}) follows from (\ref{ineg-rate-gene}) and the definition of $m_{a}(f)$. 
\end{proof}

We can prove (\ref{vit-lip-bis}). 
Property~(\ref{cor-ineg-rate-gene}), applied to $\mu_1:=\delta_x$ and $\mu_2:=\pi$ gives 
\begin{eqnarray*} 
\big|(P^nf)(x) - \pi(f)\big| &=& \big| \E[f(X_{n}^{x})] - \E[f(X_{n}^{\pi})] \big| \\
&\leq& \E\big[\, |f(X_{n}^{x}) - f(X_{n}^{\pi})|\, \big] \\ 
&\leq& \xi^{\frac{a-1}{a}}\,  m_{a}(f)\, \kappa_1^n\, \E[d(x,X_0^{\pi})]\, \big(\|\delta_x\|_a + \|\pi\|_a\big)^{a-1}. 
\end{eqnarray*}
Next observe that $\|\delta_x\|_a=p(x)$ and 
$$\E[d(x,X_0^{\pi})] \leq \E\big[d(x,x_0) + d(x_0,X_0^{\pi})\big] \leq p(x) + \pi(d(x_0,\cdot)) \leq p(x)\, \|\pi\|_1.$$
Hence $\E[d(x,X_0^{\pi})]\, (\|\delta_x\|_a + \|\pi\|_a)^{a-1} \leq  p(x)^a \|\pi\|_1 \, (1 + \|\pi\|_a)^{a-1}$. Property~(\ref{vit-lip-bis}) (namely (\ref{vit-lip}) with explicit constant) is then proved. 

Finally, to prove (\ref{vit-lip-strong-ergo}), it remains to study $m_a(P^nf)$ for $f\in\cL_a$.  Inequality~(\ref{cor-ineg-rate-gene}) applied to $\mu_1=\delta_x$ and $\mu_2=\delta_y$ gives: 
$$\forall f\in \cL_a,\ |(P^nf)(x) - (P^nf)(y)| \leq \xi^{\frac{a-1}{a}}\, m_a(f)\, \kappa_1^n\, d(x,y)\, \big(p(x) + p(y)\big)^{a-1}.$$ 
 Thus $m_a(P^nf) \leq \xi^{\frac{a-1}{a}}\, m_a(f)\, \kappa_1^n$. Since $m_a(1_{X})=0$, this gives 
$$m_a\big(P^nf-\pi(f)1_{X}\big) \leq \xi^{\frac{a-1}{a}}\,  m_a(f)\, \kappa_1^n.$$ 
Combining the last inequality with (\ref{vit-lip-bis}) gives (\ref{vit-lip-strong-ergo}). 

\section{Additional material for $P$ defined by a kernel $K$} \label{C}
Here $(X_n)_{n\in\N}$ is a Markov chain with state space $\X = \R^q\ $ ($q\in\N^*$) equipped with any norm $\|\cdot\|$, and we assume that there exists $K : \R^q\times\R^q\r[0,+\infty)$ measurable such that, for all $x\in\X$, $P(x,dy)$ is absolutely continuous with respect to the Lebesgue measure on $\R^q$, namely: $P(x,dy) = K(x,y)\, dy$. 

Let $a\in[1,+\infty)$ and $\gamma\in[0,a-1]$. The next result is useful to obtain the set inclusion $P(\cB_{\gamma}) \subset \cL_a$, where 
$\cB_{\gamma}$, $\cL_a$ are defined by (\ref{def-Bb}) and (\ref{def-La}) page \pageref{def-La}. 
\begin{apro} \label{CS}
Assume that, for all $x_0\in \R^q$, there exist a Lebesgue-integrable function $g_{x_0} : \R^q\r[0,+\infty)$  and an open neighborhood $\cU_{x_0}$ of $x_0$ in $\R^q$ such that: 
\begin{equation} \label{K-dom-no-bounded}
\forall x\in \cU_{x_0}, \quad k=1,\ldots,q, \quad \big(1+\|y\|\big)^{\gamma}\, \big| \frac{\partial K}{\partial x_k}(x,y) \big| \leq g_{x_0}(y) \ \text{ for a.e. } y\in\R^q
\end{equation} 
and assume in addition that there exists a constant $d$ such that 
\begin{equation} \label{J-dom-eta}
 k=1,\ldots,q,\quad \forall x\in\R^q, \quad \int_{\R^q} (1+\|y\|)^{\gamma}\, \big|\frac{\partial K}{\partial x_k}(x,y)\big| \, dy \leq d\, (1+\|x\|)^{a-1}.
\end{equation}
Then we have $P(\cB_{\gamma}) \subset \cL_a$. 
\end{apro}
\begin{proof}{} Let $f\in\cB_{\gamma}$. From Lebesgue's theorem, one can easily deduce that the function $Pf : \R^q\r\C$ is differentiable on $\R^q$, and that its derivative is given by: 
\begin{equation} \label{def-deri-Pf}
\forall k=0,\ldots,q,\ \forall x\in\R^q,\quad \frac{\partial (Pf)}{\partial x_k}(x) = \int_{\R^q} f(y)\, \frac{\partial K}{\partial x_k}(x,y)\, dy.
\end{equation}
For the sake of simplicity assume that $\|\cdot\|$ is the euclidean norm on $\R^q$.  By (\ref{def-deri-Pf}) and (\ref{J-dom-eta}) we obtain: $\forall x\in\R^q,\ \|\nabla(Pf)(x)\| \leq  d\sqrt{q}\, |f|_{\gamma} \big(1+\|x\|\big)^{a-1}$, where $\nabla$ stands for the gradient operator. 
Then Taylor's inequality gives for any $(x_1,x_2)\in\R^q\times \R^q$ 
\begin{eqnarray*} 
\big|Pf(x_1) - Pf(x_2)\big| &\leq& \|x_1-x_2\|\, \sup_{t\in[0,1]} \|\nabla(Pf)(tx_1+(1-t)x_2)\| \\
&\leq&  d\sqrt{q}\, |f|_{\gamma}\, \|x_1-x_2\|\, \sup_{t\in[0,1]} \big(1+\|tx_1+(1-t)x_2\|\big)^{a-1} \\
&\leq& d\sqrt{q}\, |f|_{\gamma}\, \|x_1-x_2\|\, \big(1+\|x_1\|+\|x_2\|\big)^{a-1}. 
\end{eqnarray*}
It follows that $Pf\in\cL_{a}$. 
\end{proof}

The following statement gives a simple sufficient condition for (\ref{J-dom-eta}) to hold true. 
\begin{apro}
Let $a\in[1,+\infty)$ and $V_{a-1}(\cdot) := (1+\|\cdot\|)^{a-1}$. Assume that $PV_{a-1}/V_{a-1}$ is bounded, and 
\begin{equation} \label{K'-sur-K}
M := \sup_{(x,y)\in\R^q\times \R^q} \frac{\big|(\partial_x K)(x,y)\big|}{|K(x,y)|} < \infty.
\end{equation}
Then, for each $\gamma\in[0,a-1]$, Condition~\emph{(\ref{J-dom-eta})} is fulfilled. 
\end{apro}
\begin{proof}{}
Let $x\in\R^q$. We have 
$$\int_{\R^q} (1+\|y\|)^{\gamma}\, \big|\partial_x K(x,y)\big| \, dy \leq M\, \int_{\R^q} (1+\|y\|)^{a-1}\, K(x,y)\, dy = (PV_{a-1})(x) \leq C\, V_{a-1}(x)$$
for some constant $C$. \end{proof}
\noindent{\bf Proof of Remark~\ref{d(f)-alpha}.} 
Assume that $\|\cdot\|$ is the supremum norm on $\R^q$. We obtain from (\ref{def-deri-Pf}) with $K(x,y):=\nu\big(y-Ax\big)$ and from (\ref{auto-v-v'}): $\forall f\in\cB_{\gamma},\, \forall x\in\R^q$, 
\begin{eqnarray*}
\big\|\nabla(Pf)(x)\big\| &\leq&  |f|_\gamma\, \big(\max_k\sum_{i=1}^q |a_{ik}|\big) \int_{\R^q} (1+\|y\|)^\gamma\,  \|\nabla\nu(y-Ax)\|\, dy \\
&\leq&  |f|_\gamma \, \big(\max_k\sum_{i=1}^q |a_{ik}|\big) \int_{\R^q} (1+\|y+Ax\|)^\gamma\, \frac{b}{(1+\|y\|)^\beta}\, dy \\
 &\leq& b\, C_{\gamma,\beta}\, |f|_\gamma \, \big(\max_k\sum_{i=1}^q |a_{ik}|\big)\, (1+\|x\|)^\gamma. 
\end{eqnarray*}
We easily deduce that $m_a(Pf) \leq q\, b\, C_{\gamma,\beta}\, |f|_\gamma\, \, (\max_k\sum_{i=1}^q |a_{ik}|)$. 
\fdem


\end{document}